\documentclass{article}

\usepackage[margin=1.5in]{geometry}

\usepackage{amsmath}
\usepackage{amsfonts}
\usepackage{amssymb}
\usepackage{mathtools}
\usepackage{graphicx}
\usepackage[utf8]{inputenc}
\usepackage{graphicx}
\usepackage{bmpsize}
\usepackage[font=small]{caption}
\usepackage{subcaption}
\usepackage{float}
\usepackage{epsfig}
\usepackage{multicol}
\usepackage{xcolor}

\providecommand{\U}[1]{\protect\rule{.1in}{.1in}}

\newcommand{\dt}{\Delta{t}}

\newcommand{\Div}[1]{\nabla \cdot #1}
\newcommand{\Grad}[1]{\nabla{#1}}
\newcommand{\norm}[2][]{||#2||_{#1}}

\DeclareMathOperator*{\argmax}{arg\,max}

\allowdisplaybreaks

\newtheorem{algorithm}{Algorithm}
\newtheorem{assumption}{Assumption}
\input{tcilatex}
\begin{document}
	
	\title{An embedded variable step IMEX scheme for the incompressible Navier-Stokes equations} 
	\author{
	{\sc Victor DeCaria\thanks{Oak Ridge National Laboratory, Oak Ridge, TN, USA. decariavp@ornl.gov. Partially supported by NSF grant DMS 1817542. This manuscript has been authored, in part, by UT-Battelle, LLC under Contract No. DE-AC05-00OR22725 with the U.S. Department of Energy. The United States Government retains and the publisher, by accepting the article for publication, acknowledges that the United States Government retains a non-exclusive, paid-up, irrevocable, world-wide license to publish or reproduce the published form of this manuscript, or allow others to do so, for United States Government purposes. The Department of Energy will provide public access to these results of federally sponsored research in accordance with the DOE Public Access Plan (http://energy.gov/downloads/doe-public-access-plan).} }	
	{\sc and}\\[6pt]
	{\sc Michael Schneier\thanks{University of Pittsburgh, Pittsburgh, PA, USA. mhs64@pitt.edu}.}}
	\maketitle
	
\begin{abstract}
{This report presents a series of implicit-explicit (IMEX) variable timestep algorithms for the incompressible Navier-Stokes equations (NSE). With the advent of new computer architectures there has been growing demand for low memory solvers of this type. The addition of time adaptivity improves the accuracy and greatly enhances the efficiency of the algorithm. We prove energy stability of an embedded first-second order IMEX pair. For the first order member of the pair, we prove stability for variable stepsizes, and analyze convergence. We believe this to be the first proof of this type for a variable step IMEX scheme for the incompressible NSE. We then define and test a variable stepsize, variable order IMEX scheme using these methods. Our work contributes several firsts for IMEX NSE schemes, including an energy argument and error analysis of a two-step, variable stepsize method, and embedded error estimation for an IMEX multi-step method.}
{Variable Step BE-AB2 Scheme; Implicit/Explicit; IMEX; Navier-Stokes; Time Filters}
\end{abstract}

\section{Introduction}
 Time accuracy is critical for obtaining physically relevant solutions in the field of computational fluid dynamics (CFD). Many flow solvers use constant timesteps, but there has been an expanding interest in variable step solvers \cite{JR10,KGGS10,BR12}. These methods allow for larger time steps for intervals of the simulation where the physics are stable, while allowing for smaller time steps for portions which are physically interesting. This allows for a decrease in the computational cost of the solver, while simultaneously increasing the accuracy.  

In this paper we focus on introducing several new implicit-explicit (IMEX) adaptive time stepping schemes for the incompressible Navier-Stokes equations (NSE). Methods of this type are known to be inexpensive per timestep, but often have a severe timestep restriction due to the explicit treatment of the nonlinear term. As solvers have matured and memory increased, methods of this type have seen decreased development. However, with the recent explosion of interest in uncertainty quantification and machine learning, along with newly emerging computational architectures, methods requiring less spatial, communication and computational complexity have become interesting tools again. Additionally, a prominent feature of these schemes is at each timestep they require the solution of a shifted Stokes problem. Therefore, these IMEX schemes stand to leverage recent developments of GPU solvers for the Stokes equations \cite{Z14}.

The scheme has an embedded structure, so that no additional Stokes solves or function evaluations are required to compute the second order method once the first order method is computed. This is done with an easy to implement and efficient time filter as follows. Let $u^{n}$ be a velocity approximation at $T = \Delta t n$. If $u^{n+1}$ is calculated with implicit Euler, then a second order approximation can be constructed by resetting $u^{n+1}$ with
\begin{equation}\notag
u^{n+1} \Leftarrow u^{n+1} - \frac{1}{3}(u^{n+1} - 2u^n+ u^{n-1}) \hspace{10mm} \text{(Constant stepsize timefilter}).
\end{equation}
We summarize the main contributions of this paper:
\begin{enumerate}
\item A full stability and error analysis for a first order, two step variable stepsize backward Euler - Adams Bashorth 2 (VSS BE-AB2) timestepping scheme. To our knowledge this is the first provable stability and convergence result for a two step IMEX method applied to the incompressible NSE.
\item Using a time filter, we embed a variable stepsize second order scheme into the VSS BE-AB2 algorithm, which we call VSS BE-AB2+F. We prove energy stability for the constant timesteps.
\item We combine these methods to make a variable stepsize variable order scheme, which we call multiple order, one solve, embedded IMEX - 12 (MOOSE-IMEX-12).
\end{enumerate}

These results reduce the gap between the needs of practical CFD and what analysis can contribute. A full analysis of VSS BE-AB2+F and MOOSE-IMEX-12 remains an open problem. However, numerical experiments conducted in Section \ref{sec:tests} are promising.

The paper is organized as follows. In Section \ref{sec:preliminaries}, we present preliminary analysis which will be needed in the ensuing sections. The stability and error analysis of the first order member of MOOSE-IMEX-12 is contained in Section \ref{sec:first_order}. The variable stepsize, second order method member of MOOSE-IMEX-12 is discussed in Section \ref{sec:second_order}. The full MOOSE-IMEX-12 method is described in Section \ref{sec:vsvo}. We confirm the predicted convergence rates on constant stepsize and adaptive tests in Section \ref{sec:tests}. Concluding remarks are given in Section \ref{sec:conclusion}.

\subsection{Previous Works}
Variable timestep schemes have been studied extensively for linear multistep methods for ordinary different equations (ODEs); see \cite{CL84,DLN83} and the references therein. However, there is a large gap in analysis between the fully implicit methods analyzed for ODEs, and IMEX methods which are often required for partial differential equation (PDE) based applications.
Linear stability analysis for constant timestep backward differentiation formula 2 combined with Adams Bashforth 2 (BDF2-AB2) and Crank-Nicolson Leapfrog (CNLF) applied to systems of linear evolution equations with skew symmetric couplings was conducted in \cite{LC12}. It was shown under a timestep condition that both methods were long time energy stable.  
Recently, for the NSE adaptive time stepping schemes have been studied for a variety of second order implicit and linearly implicit methods \cite{JR10,KGGS10,layton2020analysis,BR12}. It was demonstrated that time adaptivity increased the accuracy and efficiency of the schemes. A stability analysis of these methods for increasing and decreasing timestep ratio is still an open problem. Constant timestep  IMEX schemes for the NSE have been studied for Crank-Nicolson combined with Adams Bashfroth 2 (CN-AB2) \cite{JL04,MT98}, a three-step backward extrapolating scheme in \cite{BDK82}, and backward Euler-forward Euler (BE-FE) in \cite{Y08}.

\section{Notation and preliminaries\label{sec:preliminaries}}

Let $\Omega \subset \mathbb{R}^{d}, d=2,3,$ denote an open regular domain with boundary $\partial \Omega$ and let $[0,T]$ denote a time interval. We consider the incompressible NSE 
\begin{equation}\label{eq:NSE}
\left\{\begin{aligned}
u_{t}+u\cdot\nabla u-\nu\Delta u+\nabla p  &
=f(x,t)&\quad\forall x\in\Omega\times(0,T]\\
\nabla\cdot u  &  =0&\quad\forall x\in\Omega\times(0,T]\\
u  &  =0&\quad\forall x\in\partial\Omega\times(0,T]\\
u(x,0)  &  =u^{0}(x)&\quad\forall x\in\Omega.
\end{aligned}\right.
\end{equation}

We denote by $\|\cdot\|$ and $(\cdot,\cdot)$ the $L^{2}(\Omega)$ norm and inner product, respectively, and by $\|\cdot\|_{L^{p}}$ and $\|\cdot\|_{W_{p}^{k}}$ the $L^{p}(\Omega)$ and Sobolev
$W^{k}_{p}(\Omega)$ norms, 
respectively. $H^{k}(\Omega)=W_{2}^{k}(\Omega)$ with norm $\|\cdot\|_{k}$. 
The space $H^{-1}(\Omega)$ denotes the dual space of bounded linear functionals defined on $H^{1}_{0}(\Omega)=\{v\in H^{1}(\Omega)\,:\,v=0 \mbox{ on } \partial\Omega\}$; this space is equipped with the norm
$$
\|f\|_{-1}=\sup_{0\neq v\in X}\frac{(f,v)}{\| \nabla v\| } 
\quad\forall f\in H^{-1}(\Omega).
$$

We will consider a discretization of the time interval $[0,T]$ into $N$ separate intervals of varying length and define the norm
$$
\| v \|_{L^{2}(t^{n},t^{n+1}, L^{2}(\Omega))} =  \left (\int_{t_{n}}^{t_{n+1}}\|v\|_{L^{2}(\Omega)}^{2}dt\right)^{\frac{1}{2}}.
$$

The solution spaces $X$ for the velocity and $Q$ for the pressure are respectively defined as
$$
\begin{aligned}
X : =& [H^{1}_{0}(\Omega)]^{d} = \{ v \in [L^{2}(\Omega)]^{d} \,:\, \nabla v \in [L^{2}(\Omega)]^{d \times d} \ \text{and} \  v = 0 \ \text{on} \ \partial \Omega \} \\
Q : =& L^{2}_{0}(\Omega) = \Big\{ q \in L^{2}(\Omega) \,:\, \int_{\Omega} q dx = 0 \Big\}.
\end{aligned}
$$
A weak formulation of \eqref{eq:NSE} is given as follows: find $u:(0,T]\rightarrow X$ and $p:(0,T]\rightarrow Q$ such that, for almost all $t\in(0,T]$, satisfy 
\begin{equation}\label{wfwf}
\left\{\begin{aligned}
(u_{t},v)+(u\cdot\nabla u,v)+\nu(\nabla u,\nabla v)-(p
,\nabla\cdot v)  &  =(f,v)&\quad\forall v\in X\\
(\nabla\cdot u,q)  &  =0&\quad\forall q\in Q\\
u(x,0)&=u^{0}(x).&
\end{aligned}\right.
\end{equation}
The subspace of $X$ consisting of weakly divergence-free functions is defined as
$$
V :=\{v\in X \,:\,(\nabla\cdot v,q)=0\,\,\forall q\in Q\} \subset X.
$$
We denote conforming velocity and pressure finite element spaces based on a regular triangulation of $\Omega$ having maximum triangle diameter $h$ by
$
X_{h}\subset X$ {and} $ Q_{h}\subset Q.
$
We assume that the pair of spaces $(X_h,Q_h)$ satisfy the discrete inf-sup (or $LBB_h$) condition required for stability of finite element approximations; we also assume that the finite element spaces satisfy the approximation properties
$$
\begin{aligned}
\inf_{v_h\in X_h}\| v- v_h \|&\leq C h^{s+1}&\forall v\in [H^{s+1}(\Omega)]^d\\
\inf_{v_h\in X_h}\| \nabla ( v- v_h )\|&\leq C h^s&\forall v\in [H^{s+1}(\Omega)]^d\\
\inf_{q_h\in Q_h}\|  q- q_h \|&\leq C h^s&\forall q\in H^{s}(\Omega),
\end{aligned}
$$
where $C$ is a positive constant that is independent of $h$. The Taylor-Hood element pairs ($P^s$-$P^{s-1}$), $s\geq 2$, are one common choice for which the $LBB_h$ stability condition and the approximation estimates hold \cite{GR79, Max89}.

We also define the discretely divergence-free space $V_h$ as
$$
V_{h} :=\{v_{h}\in X_{h}\,:\,(\nabla\cdot v_{h},q_{h})=0\,\,\forall
q_{h}\in Q_{h}\}  \subset X.
$$

We will also assume that the mesh satisfies the following standard inverse inequalities
\begin{align}
&\|v_{h}\| \leq C h^{-1}\|\nabla v_{h} \| \ \ \ \ \ \ \ \ \ \ \ \forall v_{h} \in X_{h} \label{inv1} \\
&\|v_{h}\|_{\infty} \leq C|\ln h|^{1/2}\|\nabla v_{h}\|\ \ \  \forall v_{h} \in X_{h}, \text{for $d = 2$.} \label{inv2}
\end{align}

Since the finite elements we consider satisfy the inf-sup condition, we can use the following Lemma from \cite{GR79}.

\begin{lemma}
Suppose $(X_h,Q_h)$ satisfy the inf-sup condition. Then for all $u \in V$,
\begin{equation}\notag
\inf_{v_h\in V_h} \|\nabla(u-v_h)\| \leq C(\beta)\inf_{x_h \in X_h}\|\nabla (u- v_h)\|.
\end{equation}
\end{lemma}

We define the trilinear form
$$
b(u,v,w) = (u\cdot\nabla v,w) 
\qquad\forall u,v,w\in [H^1(\Omega)]^d,
$$
and the explicitly skew-symmetric trilinear form given by 
$$
b^{\ast}(u,v,w):=\frac{1}{2}(u\cdot\nabla v,w)-\frac{1}{2}(u\cdot\nabla w,v)
\qquad\forall u,v,w\in [H^1(\Omega)]^d \, ,
$$
or equivalently,
$$
b^{\ast}(u,v,w):=(u\cdot\nabla v,w)+\frac{1}{2}(\nabla \cdot u,v\cdot w)
\qquad\forall u,v,w\in [H^1(\Omega)]^d \,.
$$
This satisfies the bound \cite{Layton08}
\begin{gather}
b^{\ast}(u,v,w)\leq C_{b^*} \|  \nabla u\|  \| \nabla v\|  \| \nabla
w \| \qquad\forall u, v, w \in X  \label{In1} \\
b^{\ast}(u,v,w)\leq C_{b^*} (\| u \| \|  \nabla u\| )^{1/2}  \| \nabla v\|  \| \nabla
w \| \qquad\forall u, v, w \in X  \label{In3} \\
b^{\ast}(u,v,w)\leq C_{b^*} \| \nabla u\| (\| v \| \|  \nabla v\| )^{1/2} \| \nabla
w \| \qquad\forall u, v, w \in X . \label{In4}
\end{gather}
Additionally, we have the following bound
\begin{lemma}
\begin{equation}\label{In2}
b^{\ast}(u,v,w)\leq C_{b^*} \|  \nabla u\|   \| \nabla v\| (\|  w \|  \| \nabla
w \| )^{1/2}\qquad\forall u, v, w \in X. 
\end{equation}
\end{lemma}
\begin{proof}
We have by repeated H{\"o}lders inequality that
\begin{gather*}
(\nabla \cdot u,v\cdot w) =  \sum_{i=1}^d\int_{\Omega} (\nabla \cdot u)  v_iw_i dx \leq \sum_{i=1}^d\|\nabla \cdot u\|\|v_i\|_{L^6}\|w_i\|_{L^3}
\\
\leq \sqrt{d} \|\nabla \cdot u \| \|v\|_{L^6} \|w\|_{L^3} \leq C(d) \|\nabla  u \| \|v\|_{L^6} \|w\|_{L^3}.
\end{gather*}

Similarly, we have
\begin{gather*}
\int_{\Omega}(u \cdot \nabla v)\cdot w dx \leq C(d) \|u\|_{L^6} \|\nabla v\|\|w\|_{L^3}.
\end{gather*}
By Sobolev embedding theorems, $H^1\hookrightarrow L^6$ and $H^{\frac{1}{2}}\hookrightarrow L^3$ for $d=2,3$. The result then follows from the interpolation inequality $\|w\|_{\frac{1}{2}}\leq C \|w\|^{\frac{1}{2}} \|\nabla w\|^{\frac{1}{2}}$.
\end{proof}

To analyze rates of convergence in Section \ref{section:conv} we will  make the following regularity assumptions on the NSE.
\begin{assumption}\label{assumption:reg}
	In \eqref{eq:NSE} we assume
	$u^{0} \in V, \ p \in L^{2}(0,T; H^{s+1}(\Omega)),  u \in L^{\infty}(0,T;H^{1}(\Omega)) 
	\newline \cap H^{1}(0,T;H^{s+1}(\Omega)) \cap H^{2}(0,T;H^{1}(\Omega)), \text {and } f \in L^{2}(0,T;L^{2}(\Omega))$.
\end{assumption}

\section{The first order method\label{sec:first_order}}	

Our goal is to construct and analyze an IMEX version of the time filtered backward Euler method, which was analyzed for ODEs in \cite{guzel}, and for a fully implicit, constant stepsize NSE discretization in \cite{DLZ18}. These methods are based on applying a time filter to the backward Euler solution to achieve second order accuracy. Thus, we need to choose our IMEX version of backward Euler carefully.

The standard choice is the BE-FE scheme, which is 
\begin{equation}\label{BE-FE}
\begin{aligned}
\left(\frac{u_{h}^{n+1} - u_{h}^{n}}{\dt},v_h\right) + \nu(\nabla u^{n+1}_h,\nabla v_h) + &b^{\ast}(u_{h}^{n},u_{h}^{n},v_h)  
\\ - (p^{n+1},\nabla \cdot v_h)  &= (f^{n+1},v_{h}) \qquad \qquad \forall
v_{h}\in X_{h}
\\
(\nabla \cdot u_{h}^{n+1},q_{h}) &= 0 \qquad \qquad \qquad \qquad \forall
q_{h}\in Q_{h}.
\end{aligned}
\end{equation}
This is insufficient, since the time filter will not correct the first order, explicit treatment of the nonlinearity. Instead, we use a nonstandard BE-AB2 combination where the constant extrapolation $u_h^{n+1} = u_{h}^n + \mathcal{O}(\Delta t)$ in the nonlinearity is replaced with a linear extrapolation. For constant stepsizes, this means $u^{n+1}_h \approx 2u^{n}_h-u^{n-1}_h + \mathcal{O}(\Delta t^2)$.

For variable stepsizes, let $\Delta t_n = t^{n+1}-t^n$. The stepsize ratios are $\omega_n = \frac{\Delta t_n}{\Delta t_{n-1}}$. The second order extrapolation of $u_h^{n+1}$ becomes $E^{n+1}(u_h) :=(1+\omega_n)u^{n}_h - \omega_n u^{n-1}_h$. We then have the variable stepsize BE-AB2 (VSS BE-AB2) method.

\begin{equation}\label{eqn:VSS-BE-AB2}
\begin{aligned}
\left(\frac{u_{h}^{n+1} - u_{h}^{n}}{\dt_n},v_h\right) + \nu(\nabla u^{n+1}_h,\nabla v_h) + &b^{\ast}(E^{n+1}(u_h),E^{n+1}(u_h),v_h)  
\\ - (p^{n+1},\nabla \cdot v_h)  &= (f^{n+1},v_{h}) \qquad \qquad \forall
v_{h}\in X_{h}
\\
(\nabla \cdot u_{h}^{n+1},q_{h}) &= 0 \qquad \qquad \qquad \qquad \forall
q_{h}\in Q_{h}.
\end{aligned}
\end{equation}

This is a second order perturbation of implicit backward Euler, and applying the time filter results in a second order method. In the next two subsections, we rigorously show that this new method is variable stepsize stable, and is globally convergent.

\subsection{Energy Stability for VSS BE-AB2\label{section:stab}}

In this section we prove nonlinear, conditional stability of \eqref{eqn:VSS-BE-AB2}. We begin with a general stability result. We then show that the timestep condition can be improved in some special cases.

\begin{theorem}\label{theorem:stability_BE}[General Stability of VSS BE-AB2]
Consider the method \eqref{eqn:VSS-BE-AB2}, let $\Omega \subset \mathbb{R}^{d}, d =2,3,$ and $C_{stab} > 0$ be a constant independent of $h,\Delta t_{n}, \omega_{n}, \nu$ and $u$. Suppose that
\begin{equation}\label{vssstab_cond_1}
1 - \frac{C_{stab} \dt_n(1+\omega_n^2)}{\nu h}\| \nabla E^{n+1}(u_h)\|^{2} \geq 0.
\end{equation}
Then, for any $N> 1$
\begin{equation}\label{energy_inequality}
\begin{aligned}
&\frac{1}{2}\|u_{h}^{N}\|^{2} + \frac{1}{4}\|u_{h}^{N} - u_{h}^{N-1}\|^{2} +  \frac{\nu }{4}\sum_{n=1}^{N-1}\dt_n\|\nabla u_{h}^{n+1}\|^{2} 
\\ 
&+ \sum_{n=1}^{N-1}\frac{1}{8(1+\omega_n^2)}\|  u^{n+1}_{h} - u_{h}^{n} +\omega_n(u_{h}^{n} - u_{h}^{n-1})\|^{2} \leq \sum_{n=1}^{N-1} \frac{\dt_n}{\nu}\|f^{n+1}\|^{2}_{-1} 
\\
&+ \frac{1}{2}\|u^{1}_{h}\|^{2} + \frac{1}{4}\|u_{h}^{1} - u_{h}^{0}\|^{2}.
\end{aligned}
\end{equation}
\end{theorem}
\begin{proof}
	Setting $v_h = u_h^{n+1}$ and multiplying by $\dt_n$ we have
	\begin{gather*}
	\frac{1}{2}\|u_{h}^{n+1}\|^{2} - \frac{1}{2}\|u_{h}^{n}\|^{2} +  \frac{1}{2}\|u_{h}^{n+1} - u_{h}^{n}\|^{2} + \dt_n\nu\|\nabla u_{h}^{n+1}\|^{2}
	\\
	+ \dt_n b^{\ast}(E^{n+1}(u_h),E^{n+1}(u_h),u_{h}^{n+1}) = \dt_n (f^{n+1},u_{h}^{n+1}). 
	\end{gather*}
	Applying Young's inequality to the right hand side then gives
	\begin{gather*}
	\frac{1}{2}\|u_{h}^{n+1}\|^{2} - \frac{1}{2}\|u_{h}^{n}\|^{2}+ \frac{1}{2}\|u_{h}^{n+1} - u_{h}^{n}\|^{2} + \dt_n\nu\|\nabla u_{h}^{n+1}\|^{2}
	\\
	\dt_n b^{\ast}(u_{h}^n+ \omega_n (u_{h}^{n}-u_{h}^{n-1}),u_{h}^n+ \omega_n (u_{h}^{n}-u_{h}^{n-1}),u_{h}^{n+1}) 
	\\
	\leq \frac{\nu \dt_n}{4}\|\nabla u_{h}^{n+1}\|^{2} + \frac{\dt_n}{\nu}\|f^{n+1}\|_{-1}^{2}. 
	\end{gather*}
	Next, we deal with the nonlinearity. Applying \eqref{In2}, using the skew symmetry of the nonlinearity, applying the Cauchy-Schwarz-Young, Poincar\'{e}-Friedrichs, and inverse inequalities we have
	\begin{equation*}
	\begin{aligned}
	&\dt_n b^{\ast}(E^{n+1}(u_h),u_{h}^n+ \omega_n (u_{h}^{n}-u_{h}^{n-1}),u_{h}^{n+1})  \\ 
	&= \dt_n b^{\ast}(E^{n+1}(u_h),u_{h}^{n+1},u_{h}^{n+1}-u_{h}^n- \omega_n (u_{h}^{n}-u_{h}^{n-1})) \\
	&\leq C \dt_n h^{-\frac{1}{2}}\| \nabla E^{n+1}(u_h)\|\| u_{h}^{n+1}-u_{h}^n- \omega_n (u_{h}^{n}-u_{h}^{n-1}) \|\| \nabla u^{n+1}_h\| \\
	&\leq C \frac{\dt_n^{2}(1+\omega_n^2)}{h}\| \nabla E^{n+1}(u_h)\|^{2}\| \nabla u^{n+1}_h\|^{2} \\
	&+ \frac{1}{8(1+\omega_n^2)}\| u_{h}^{n+1}-u_{h}^n- \omega_n (u_{h}^{n}-u_{h}^{n-1}) \|^{2}.
	\end{aligned}
	\end{equation*}
	For the last term we have by the parallelogram law
	\begin{gather*}
	 \frac{1}{8(1+\omega_n^2)}\| u_{h}^{n+1}-u_{h}^n- \omega_n (u_{h}^{n}-u_{h}^{n-1}) \|^2 
	\\
	=\frac{1}{4(1+\omega_n^2)}\| u^{n+1}_{h} - u_{h}^{n}\|^{2} + \frac{\omega_n^2}{4(1+\omega_n^2)}\|u_{h}^{n} - u_{h}^{n-1}\|^{2} 
	\\
	- \frac{1}{8(1+\omega_n^2)}\|  u^{n+1}_{h} - u_{h}^{n} +\omega_n(u_{h}^{n} - u_{h}^{n-1})\|^{2}
	\\
	\leq \frac{1}{4}\| u^{n+1}_{h} - u_{h}^{n}\|^{2} + \frac{1}{4}\|u_{h}^{n} - u_{h}^{n-1}\|^{2}
	\\
	- \frac{1}{8(1+\omega_n^2)}\|  u^{n+1}_{h} - u_{h}^{n} +\omega_n(u_{h}^{n} - u_{h}^{n-1})\|^{2}.
	\end{gather*}
	Combining like terms we then have
	\begin{gather*}
	\frac{1}{2}\|u_{h}^{n+1}\|^{2} - \frac{1}{2}\|u_{h}^{n}\|^{2} + \frac{1}{4}\|u_{h}^{n+1} - u_{h}^{n}\|^{2} - \frac{1}{4}\|u_{h}^{n} - u_{h}^{n-1}\|^{2}  \\
	+ \frac{\nu \dt_n}{4} \|\nabla u_{h}^{n+1} \|^{2} +  \frac{\nu \dt_n}{2}\left(1 - \frac{C \dt_n(1+\omega_n^2)}{\nu h}\| \nabla E^{n+1}(u_h)\|^{2}\right)\|\nabla u_{h}^{n+1}\|^{2} 
	\\
	 + \frac{1}{8(1+\omega_n^2)}\|  u^{n+1}_{h} - u_{h}^{n} +\omega_n(u_{h}^{n} - u_{h}^{n-1})\|^{2} \leq \frac{\dt_n}{\nu}\|f^{n+1}\|^{2}_{-1}.
	\end{gather*}
	Finally, using condition \eqref{vssstab_cond_1}, letting $C = C_{stab}$, and summing from $n = 1$ to $N-1$ the result follows. 
\end{proof}	

There are several cases where the time step condition can be improved by using a different embedding for the nonlinear term. When $\Omega \subset \mathbb{R}^{2}$ the discrete Sobolev embedding will give a less restrictive timestep condition compared to that in Theorem \ref{theorem:stability_BE}.

\begin{theorem}\label{2d stability }[2d Stability of VSS BE-AB2]
	Consider the method \eqref{eqn:VSS-BE-AB2} and let $\Omega \subset \mathbb{R}^{2}$. Suppose that
	\begin{equation}\label{vssstab_cond_2}
	1 - \frac{C_{stab} \dt_n(1+\omega_n^2)|\ln h|}{\nu}\| \nabla E^{n+1}(u_h)\|^{2} \geq 0.
	\end{equation}
	Then, the energy inequality, \eqref{energy_inequality}, from Theorem \ref{theorem:stability_BE} holds.
\end{theorem}
\begin{proof}
	The proof is similar to that of Theorem \ref{theorem:stability_BE}, the key difference being in the treatment of the nonlinearity.  Using Holders inequality for the nonlinear term we have
	\begin{equation*}
	\begin{aligned}
	&\dt_n b^{\ast}(E^{n+1}(u_h),u_{h}^{n+1},u_{h}^{n+1}-u_{h}^n- \omega_n (u_{h}^{n}-u_{h}^{n-1}))
	\\
	&\leq C \dt_n \|E^{n+1}(u_h)\|_{\infty} \|u_{h}^{n+1}-u_{h}^n- \omega_n (u_{h}^{n}-u_{h}^{n-1}) \| \|\nabla u_{h}^{n+1}\|
	\\ 
	&+ C \frac{\dt_n}{2}  \|\nabla \cdot E^{n+1}(u_h)\| \|u_{h}^{n+1}-u_{h}^n- \omega_n (u_{h}^{n}-u_{h}^{n-1}) \| \|u_{h}^{n+1}\|_{\infty}.
	\end{aligned}
	\end{equation*}
	Then, applying \eqref{inv2} and Cauchy-Schwarz-Young
	\begin{equation*}
	\begin{aligned}
	&C \dt_n \|E^{n+1}(u_h)\|_{\infty} \|u_{h}^{n+1}-u_{h}^n- \omega_n (u_{h}^{n}-u_{h}^{n-1}) \| \|\nabla u_{h}^{n+1}\|
\\ 
&+ C \frac{\dt_n}{2}  \|\nabla \cdot E^{n+1}(u_h)\| \|u_{h}^{n+1}-u_{h}^n- \omega_n (u_{h}^{n}-u_{h}^{n-1}) \| \|u_{h}^{n+1}\|_{\infty}
	\\
	&\leq C |\ln h|^{1/2} \dt_{n} \|\nabla E^{n+1}(u_h)\|\|u_{h}^{n+1}-u_{h}^n- \omega_n (u_{h}^{n}-u_{h}^{n-1}) \| \|\nabla u_{h}^{n+1}\|
	\\
	&\leq C {\dt_n^{2}(1+\omega_n^2) |\ln h|}\| \nabla E^{n+1}(u_h)\|^{2}\| \nabla u^{n+1}_h\|^{2} 
	\\
	&+ \frac{1}{8(1+\omega_n^2)}\| u_{h}^{n+1}-u_{h}^n- \omega_n (u_{h}^{n}-u_{h}^{n-1}) \|^{2}.
	\end{aligned}
	\end{equation*}
	The result then follows from Theorem \ref{theorem:stability_BE}.
\end{proof}	

 We also derive stability estimates that do not involve the full gradient of the solution. 

\begin{theorem}\label{L3 stability}[L3/L6 Stability of VSS BE-AB2]
	Consider the method \eqref{eqn:VSS-BE-AB2} and let $\Omega \subset \mathbb{R}^{d}, d =2,3$. Suppose that
	\begin{equation}\label{vssstab_cond_3}
	1 - \frac{C_{stab} \dt_n(1+\omega_n^2)}{\nu h^{2}}\|E^{n+1}(u_h)\|_{L^3}^{2} \geq 0,
	\end{equation}
	or
	\begin{equation}\label{vssstab_cond_4}
	1 - \frac{C_{stab} \dt_n(1+\omega_n^2)}{\nu h}\left(\|E^{n+1}(u_h)\|^{2}_{L^{6}} + \|\nabla \cdot E^{n+1}(u_h)\|^{2} \right) \geq 0.
	\end{equation}
		Then, the energy inequality, \eqref{energy_inequality}, from Theorem \ref{theorem:stability_BE} holds.
\end{theorem}
\begin{proof}
	Proving \eqref{vssstab_cond_3} first, using Holders inequality for the nonlinear term we have
	\begin{equation*}
	\begin{aligned}
	&\dt_n b^{\ast}(E^{n+1}(u_h),u_{h}^{n+1},u_{h}^{n+1}-u_{h}^n- \omega_n (u_{h}^{n}-u_{h}^{n-1}) 
	\\
	&\leq C \frac{\dt_n}{2} \|E^{n+1}(u_h)\|_{L^{3}} \|u_{h}^{n+1}-u_{h}^n- \omega_n (u_{h}^{n}-u_{h}^{n-1}) \|_{L^{6}} \|\nabla u_{h}^{n+1}\|
	\\ 
	&+ C \frac{\dt_n}{2}  \|E^{n+1}(u_h)\|_{L^{3}} \|\nabla(u_{h}^{n+1}-u_{h}^n- \omega_n (u_{h}^{n}-u_{h}^{n-1})) \| \|u_{h}^{n+1}\|_{L^{6}}.
	\end{aligned}
	\end{equation*}
	Using Sobolev embeddings and the inverse inequality we then have
	\begin{equation*}
	\begin{aligned}
	&\|u_{h}^{n+1}\|_{L^{6}} \leq C \|\nabla u_{h}^{n+1}\| 
	\\
	&\|u_{h}^{n+1}-u_{h}^n- \omega_n (u_{h}^{n}-u_{h}^{n-1}) \|_{L^{6}} \leq C h^{-1} \|u_{h}^{n+1}-u_{h}^n- \omega_n (u_{h}^{n}-u_{h}^{n-1}) \|
	\\
	& \|\nabla(u_{h}^{n+1}-u_{h}^n- \omega_n (u_{h}^{n}-u_{h}^{n-1})) \| \leq C h^{-1} \|u_{h}^{n+1}-u_{h}^n- \omega_n (u_{h}^{n}-u_{h}^{n-1}) \|.
	\end{aligned}
	\end{equation*}
	Applying these inequalities and Cauchy-Schwarz-Young it follows
	\begin{equation*}
	\begin{aligned}
&C \frac{\dt_n}{2} \|E^{n+1}(u_h)\|_{L^{3}} \|u_{h}^{n+1}-u_{h}^n- \omega_n (u_{h}^{n}-u_{h}^{n-1}) \|_{L^{6}} \|\nabla u_{h}^{n+1}\|
\\ 
&+ C \frac{\dt_n}{2}  \|E^{n+1}(u_h)\|_{L^{3}} \|\nabla(u_{h}^{n+1}-u_{h}^n- \omega_n (u_{h}^{n}-u_{h}^{n-1})) \| \|u_{h}^{n+1}\|_{L^{6}}
\\
&\leq \frac{C {\dt_n^{2} (1+\omega_{n}^{2})}}{h^{2}}\|E^{n+1}(u_h)\|_{L^{3}}\|\nabla u_{h}^{n+1} \|^{2} 
\\
&+ \frac{1}{8(1+\omega_n^2)}\| u_{h}^{n+1}-u_{h}^n- \omega_n (u_{h}^{n}-u_{h}^{n-1}) \|^{2}. 
	\end{aligned}
	\end{equation*}
	The energy inequality, \eqref{energy_inequality}, then follows from condition \eqref{vssstab_cond_3} and the proof of Theorem \ref{theorem:stability_BE}. 
	
	Turning to \eqref{vssstab_cond_4}, we again use Holders inequality for the nonlinear term
	\begin{equation*}
		\begin{aligned}
		&\dt_n b^{\ast}(E^{n+1}(u_h),u_{h}^{n+1},u_{h}^{n+1}-u_{h}^n- \omega_n (u_{h}^{n}-u_{h}^{n-1}) 
		\\
		&\leq C \frac{\dt_n}{2} \|E^{n+1}(u_h)\|_{L^{6}} \|u_{h}^{n+1}-u_{h}^n- \omega_n (u_{h}^{n}-u_{h}^{n-1}) \|_{L^{3}} \|\nabla u_{h}^{n+1}\|
		\\ 
		&+ C \frac{\dt_n}{2}  \|\nabla \cdot E^{n+1}(u_h)\| \|(u_{h}^{n+1}-u_{h}^n- \omega_n (u_{h}^{n}-u_{h}^{n-1})) \|_{L^{3}} \|u_{h}^{n+1}\|_{L^{6}}.
		\end{aligned}
		\end{equation*}
		Using Sobolev embeddings we have
		\begin{equation*}
		\begin{aligned}
		&\|u_{h}^{n+1}\|_{L^{6}} \leq C \|\nabla u_{h}^{n+1}\| 
		\\
		&\|u_{h}^{n+1}-u_{h}^n- \omega_n (u_{h}^{n}-u_{h}^{n-1}) \|_{L^{3}} \leq C h^{-1/2} \|u_{h}^{n+1}-u_{h}^n- \omega_n (u_{h}^{n}-u_{h}^{n-1}) \|.
		\end{aligned}
		\end{equation*}
		Applying these inequalities and Cauchy-Schwarz-Young it follows
		\begin{equation*}
		\begin{aligned}
		&C \frac{\dt_n}{2} \|E^{n+1}(u_h)\|_{L^{6}} \|u_{h}^{n+1}-u_{h}^n- \omega_n (u_{h}^{n}-u_{h}^{n-1}) \|_{L^{3}} \|\nabla u_{h}^{n+1}\|
		\\ 
		&+ C \frac{\dt_n}{2}  \|\nabla \cdot E^{n+1}(u_h)\| \|(u_{h}^{n+1}-u_{h}^n- \omega_n (u_{h}^{n}-u_{h}^{n-1})) \|_{L^{3}} \|u_{h}^{n+1}\|_{L^{6}}
		\\
		&\leq \frac{C {\dt_n^{2} (1+\omega_{n}^{2})}}{h}\left(\|E^{n+1}(u_h)\|^{2}_{L^{6}} + \|\nabla \cdot E^{n+1}(u_h)\|^{2} \right)\|\nabla u_{h}^{n+1} \|^{2} 
		\\
		&+ \frac{1}{8(1+\omega_n^2)}\| u_{h}^{n+1}-u_{h}^n- \omega_n (u_{h}^{n}-u_{h}^{n-1}) \|^{2}. 
		\end{aligned}
		\end{equation*}
		The energy inequality, \eqref{energy_inequality}, then follows from condition \eqref{vssstab_cond_4} and the proof of Theorem \ref{theorem:stability_BE}.
\end{proof}	
\begin{remark}
	When divergence free elements such as Scott-Vogelius  are used stability condition \eqref{vssstab_cond_4} will depend only on the $L^{6}$ norm of the solution. 
\end{remark}

\subsection{Error Analysis\label{section:conv}}
For this section, we consider the fully-discrete variable step scheme \eqref{eqn:VSS-BE-AB2}. Assuming that $LBB_h$ is satisfied, algorithm \eqref{eqn:VSS-BE-AB2} is equivalent to
\begin{equation}\label{eqn:VSS-BE-AB2-LBBh}
\begin{aligned}
\left(\frac{u_{h}^{n+1} - u_{h}^{n}}{\dt_n},v_h\right) + \nu(\nabla u^{n+1}_h,\nabla v_h) + &b^{\ast}(E^{n+1}(u_h),E^{n+1}(u_h),v_h)  
\\ &= (f^{n+1},v_{h}) \qquad \qquad \forall
v_{h}\in V_{h}.
\end{aligned}
\end{equation}
For the error analysis we will split the velocity as follows
\begin{equation*}
e^{n+1}_u := u^{n+1}-u^{n+1}_h = \left(u^{n+1}-I_{h}(u^{n+1})\right)-\left(u^{n+1}_h-I_{h}(u^{n+1})\right) := \eta^{n+1}-\varphi^{n+1}_h,
\end{equation*}
where $I_{h}$ is the $L^{2}$ projection into the discretely divergence-free space $V_{h}$.

We begin by giving estimates for the consistency errors.

\begin{lemma}[Consistency Error]\label{lemma:consitency_error}
	For $u$ satisfying the regularity assumptions in Assumption \ref{assumption:reg} the following inequalities hold
	\begin{equation}\label{eqn:extrap_consist}
	\begin{aligned}
	&\left\|\frac{u^{n+1}-u^{n}}{\dt_{n}} - u_{t}^{n+1} \right\|^2 \leq C \Delta t_{n}^2 \| \nabla u_{tt}\|^{2}_{L^{2}(t^{n-1},t^{n+1} L^{2}(\Omega))}\\
	&\|\nabla(u^{n+1}-E^{n+1}(u))\|^2 \leq C (\Delta t_{n-1}+ \Delta t_{n})^3 \| \nabla u_{tt}\|^{2}_{L^{2}(t^{n-1},t^{n+1} L^{2}(\Omega))}.
\end{aligned}
	\end{equation}
\end{lemma}
\begin{proof}
	This follows by Taylor's Theorem with integral remainder.
	\end{proof}

We now prove an estimate for the nonlinear terms which will appear in the error analysis.

\begin{lemma}\label{lemma:nonlinear}[Estimate on the Nonlinear Term]
	For $u$ satisfying the regularity assumptions in Assumption \ref{assumption:reg} the following inequality holds for the nonlinear term
	\begin{equation*}
	\begin{aligned}
	&b^*\left(u^{n+1},u^{n+1},\varphi^{n+1}_h\right)-b^{\ast}(E^{n+1}(u_h),E^{n+1}(u_h),\varphi^{n+1}_h) 
	\\
	& \leq \frac{5}{64}\nu\norm{\nabla \varphi_{h}^{n+1}}^{2} + \frac{C}{\nu}\norm{\nabla E^{n+1}(\eta)}^{2} + \frac{C (\Delta t_{n-1}+ \Delta t_{n})^3}{\nu} \| \nabla u_{tt}\|^{2}_{L^{2}(t^{n-1},t^{n+1} L^{2}(\Omega))}  
	\\
	& + C \nu^{-3} \norm{\varphi_{h}^{n}}^{2} + C \nu^{-3} \norm{\varphi_{h}^{n-1}}^{2} + \frac{1}{16}\nu \norm{\nabla \varphi_{h}^{n}}^{2} + \frac{1}{16}\nu\norm{\nabla \varphi_{h}^{n-1}}^{2} 
	\\
	& + \frac{C}{\nu}\norm{E^{n+1}(u_h)}\norm{\nabla E^{n+1}(u_h)}\norm{\nabla \eta^{n+1}}^{2}  + \frac{1}{4\dt_{n}}\| \varphi^{n+1}_{h} - \varphi_{h}^{n}\|^{2}
	\\
	&+ \frac{1}{4\dt_{n}}\|\varphi_{h}^{n} - \varphi_{h}^{n-1}\|^{2} +\frac{5 C_{stab} \dt_n(1+\omega_n^2)}{8 h} \norm{\nabla E^{n+1}(u_h)}^{2} \norm{\nabla \varphi^{n+1}_h}^{2} 
	\\
	&+ \frac{Ch(\Delta t_{n-1}+ \Delta t_{n})^3}{C_{stab} \dt_n(1+\omega_n^2)} \left(\| \nabla u_{tt}\|^{2}_{L^{2}(t^{n-1},t^{n+1} L^{2}(\Omega))} + \| \nabla \eta_{tt}\|^{2}_{L^{2}(t^{n-1},t^{n+1} L^{2}(\Omega))}\right).
	\end{aligned}
	\end{equation*}
\end{lemma}
\begin{proof}
	Adding and subtracting $b^{\ast}(u_{h}^{n+1},u^{n+1},\varphi_{h}^{n+1})$, $b^{\ast}(E^{n+1}(u_h),u^{n+1},\varphi_{h}^{n+1})$, $b^{\ast}(E^{n+1}(u),u^{n+1},\varphi_{h}^{n+1})$, and $b^{\ast}(u^{n+1},u^{n+1},\varphi_{h}^{n+1})$  we have
	\begin{equation}\label{lemnon:eq1}
	\begin{aligned}
	& b^*\left(u^{n+1},u^{n+1},\varphi^{n+1}_h\right)-b^{\ast}(E^{n+1}(u_h),E^{n+1}(u_h),\varphi^{n+1}_h) \\
	&= b^{\ast}(E^{n+1}(e_{u}),u^{n+1},\varphi^{n+1}_h)
	+b^{\ast}(u^{n+1} - E^{n+1}(u),u^{n+1 },\varphi^{n+1}_h)\\
	&+ b^{\ast}(E^{n+1}(u_h),u^{n+1} - E^{n+1}(u_h),\varphi^{n+1}_h).
	\end{aligned}
	\end{equation}
	For the first term on the right hand side of \eqref{lemnon:eq1} we split it into
	\begin{equation*} b^{\ast}(E^{n+1}(e_{u}),u^{n+1},\varphi^{n+1}_h) = b^{\ast}(E^{n+1}(\eta),u^{n+1},\varphi^{n+1}_h) + b^{\ast}(E^{n+1}(\varphi_h),u^{n+1},\varphi^{n+1}_h) .
	\end{equation*}
	Applying Cauchy-Schwarz-Young, inequality \eqref{In1}, and Assumption \ref{assumption:reg}
	\begin{equation*}
	b^{\ast}(E^{n+1}(\eta),u^{n+1},\varphi^{n+1}_h) \leq \nu C_{1}\norm{\nabla \varphi_{h}^{n+1}}^{2} + \frac{C}{\nu}\norm{\nabla E^{n+1}(\eta)}^{2}.
	\end{equation*}
	Next, we have
	\begin{equation*}
	b^{\ast}(E^{n+1}(\varphi_h),u^{n+1},\varphi^{n+1}_h) = b^{\ast}((1+\omega_{n})\varphi^{n}_h,u^{n+1},\varphi^{n+1}_h) + b^{\ast}(\omega_{n}\varphi^{n-1}_h,u^{n+1},\varphi^{n+1}_h).
	\end{equation*} 
	Using inequality \eqref{In3}, Cauchy-Schwarz-Young, and  Assumption \ref{assumption:reg} we have
	\begin{equation*}
	\begin{aligned}
	b^{\ast}&((1+\omega_{n})\varphi^{n}_h,u^{n+1},\varphi^{n+1}_h) \leq C \norm{\nabla\varphi^{n}_h}^{1/2}\norm{\varphi^{n}_h}^{1/2}\norm{\nabla \varphi^{n+1}_h}
	\\ 
	&\leq C\left(\epsilon\norm{\nabla \varphi^{n+1}_h}^{2} + \frac{1}{\epsilon}\norm{\nabla\varphi^{n}_h}\norm{\varphi^{n}_h}\right)		
	\\
	&\leq C\left(\epsilon\norm{\nabla \varphi^{n+1}_h}^{2} + \frac{1}{\epsilon}\left(\alpha\norm{\nabla\varphi^{n}_h}^{2} + \frac{1}{\alpha}\norm{\varphi^{n}_h}\right)\right)	
	\\
	&\leq  \nu C_{2} \norm{\nabla \varphi_{h}^{n+1}}^{2} + \nu C_{3} \norm{\nabla \varphi_{h}^{n}}^{2} + C \nu^{-3} \norm{\varphi_{h}^{n}}^{2}.
	\end{aligned}
	\end{equation*}
	Similarly,
	\begin{equation*}
	b^{\ast}(\omega_{n}\varphi^{n-1}_h,u^{n+1},\varphi^{n+1}_h) \leq \nu  C_{4} \norm{\nabla \varphi_{h}^{n+1}}^{2} + \nu C_{5} \norm{\nabla \varphi_{h}^{n-1}}^{2} + C \nu^{-3} \norm{\varphi_{h}^{n-1}}^{2}.
	\end{equation*}
	Bounding the second nonlinear term on the right hand side of \eqref{lemnon:eq1}  using Cauchy-Schwarz-Young, inequality \eqref{In1}, Lemma \ref{lemma:consitency_error}, and Assumption \ref{assumption:reg}
	\begin{equation*}
	\begin{aligned}
	b^{\ast}&(u^{n+1} - E^{n+1}(u),u^{n+1 },\varphi^{n+1}_h) 
	\\
	&\leq \frac{C(\Delta t_{n-1}+ \Delta t_{n})^3}{\nu} \| \nabla u_{tt}\|^{2}_{L^{2}(t^{n-1},t^{n+1} L^{2}(\Omega))} + C_{6}\nu\norm{\nabla \varphi^{n+1}_h}^{2}.
	\end{aligned}
	\end{equation*}
	For the last nonlinear term in \eqref{lemnon:eq1},  adding and subtracting \\$b^{\ast}(E^{n+1}(u_h),u_{h}^{n+1 },\varphi^{n+1}_h)$, and using the skew-symmetry of the nonlinear term yields
	\begin{equation}\label{lemnon:eqn2}
	\begin{aligned}
	b^{\ast}&(E^{n+1}(u_h),u^{n+1} - E^{n+1}(u_h),\varphi^{n+1}_h)  	
	\\
	&= b^{\ast}(E^{n+1}(u_h),e_{u}^{n+1},\varphi^{n+1}_h) + b^{\ast}(E^{n+1}(u_h),u_{h}^{n+1} - E^{n+1}(u_h),\varphi^{n+1}_h) 
	\\
	& = b^{\ast}(E^{n+1}(u_h),\eta^{n+1},\varphi^{n+1}_h) + b^{\ast}(E^{n+1}(u_h),u_{h}^{n+1} - E^{n+1}(u_h),\varphi^{n+1}_h). 
	\end{aligned}
	\end{equation}
	For the first term on the right hand side we have  by Cauchy-Schwarz-Young and inequality \eqref{In3}
	\begin{equation*}
	\begin{aligned}
	b^{\ast}&(E^{n+1}(u_h),\eta^{n+1},\varphi^{n+1}_h) 
	\\
	&\leq C\norm{E^{n+1}(u_h)}^{1/2}\norm{\nabla E^{n+1}(u_h)}^{1/2}\norm{\nabla \eta^{n+1}}\norm{\nabla \varphi^{n+1}_{h}}
	\\
	&\leq \frac{C}{\nu}\norm{E^{n+1}(u_h)}\norm{\nabla E^{n+1}(u_h)}\norm{\nabla \eta^{n+1}}^{2} + C_{7}\norm{\nabla \varphi^{n+1}_{h}}^{2}.
	\end{aligned}
	\end{equation*}
	For the second term on the right hand side of \eqref{lemnon:eqn2} we rewrite it as 
	\begin{equation*}
	\begin{aligned}
	b^{\ast}&(E^{n+1}(u_h),u_{h}^{n+1} - E^{n+1}(u_h),\varphi^{n+1}_h) =
	\\ &b^{\ast}(E^{n+1}(u_h),u^{n+1} - E^{n+1}(u),\varphi^{n+1}_h) - b^{\ast}(E^{n+1}(u_h),e_{u}^{n+1} - E^{n+1}(e_{u}),\varphi^{n+1}_h).  
	\end{aligned}
	\end{equation*}
	Bounding the first of these terms using Cauchy-Schwarz-Young, inequality \eqref{In1}, and Lemma \ref{lemma:consitency_error}
	\begin{equation*}
	\begin{aligned}
	b^{\ast}&(E^{n+1}(u_h),u^{n+1} - E^{n+1}(u),\varphi^{n+1}_h) 
	\\
	&\leq C \norm{\nabla E^{n+1}(u_h)} \norm{\nabla(u^{n+1} - E^{n+1}(u))} \norm{\nabla \varphi^{n+1}_h}
	\\
	&\leq \frac{C_{stab} \dt_n(1+\omega_n^2)}{16h}\norm{\nabla E^{n+1}(u_h)}^{2} \norm{\nabla \varphi^{n+1}_h}^{2}
	\\
	&+ \frac{Ch(\Delta t_{n-1}+ \Delta t_{n})^3}{C_{stab} \dt_n(1+\omega_n^2)} \| \nabla u_{tt}\|^{2}_{L^{2}(t^{n-1},t^{n+1} L^{2}(\Omega))}.
	\end{aligned}
	\end{equation*}
	We next rewrite the term
	\begin{equation*}
	\begin{aligned}
	b^{\ast}&(E^{n+1}(u_h),e_{u}^{n+1} - E^{n+1}(e_{u}),\varphi^{n+1}_h) 
	\\&= b^{\ast}(E^{n+1}(u_h),\eta^{n+1} - E^{n+1}(\eta),\varphi^{n+1}_h) + b^{\ast}(E^{n+1}(u_h),\varphi_{h}^{n+1} - E^{n+1}(\varphi_{h}),\varphi^{n+1}_h).
	\end{aligned}
	\end{equation*}
		Bounding the first of these terms using Cauchy-Schwarz-Young,inequality \eqref{In1}, and Lemma \ref{lemma:consitency_error}
	\begin{equation*}
	\begin{aligned}
	b^{\ast}&(E^{n+1}(u_h),\eta^{n+1} - E^{n+1}(\eta),\varphi^{n+1}_h) 
	\\
	&\leq C \norm{\nabla E^{n+1}(u_h)} \norm{\nabla(\eta^{n+1} - E^{n+1}(\eta))} \norm{\nabla \varphi^{n+1}_h}
	\\
	&\leq \frac{C_{stab} \dt_n(1+\omega_n^2)}{16h}\norm{\nabla E^{n+1}(u_h)}^{2} \norm{\nabla \varphi^{n+1}_h}^{2}
	\\
	&+ \frac{Ch(\Delta t_{n-1}+ \Delta t_{n})^3}{C_{stab} \dt_n(1+\omega_n^2)} \| \nabla \eta_{tt}\|^{2}_{L^{2}(t^{n-1},t^{n+1} L^{2}(\Omega))}.
	\end{aligned}
	\end{equation*} 
	Finally, bounding the last term using Cauchy-Schwarz-Young, inequality \eqref{In4}, and the inverse inequality
	\begin{equation}
	\begin{aligned}
	b^{\ast}&(E^{n+1}(u_h),\varphi_{h}^{n+1} - E^{n+1}(\varphi_{h}),\varphi^{n+1}_h)
	\\
	&\leq C \norm{\nabla E^{n+1}(u_h)} \norm{\varphi_{h}^{n+1} - E^{n+1}(\varphi_{h})}^{1/2} \norm{\nabla (\varphi_{h}^{n+1} - E^{n+1}(\varphi_{h}))}^{1/2} \norm{\nabla \varphi^{n+1}_h}
	\\
	&\leq C h^{-\frac{1}{2}} \norm{\nabla E^{n+1}(u_h)} \norm{\varphi_{h}^{n+1} - E^{n+1}(\varphi_{h})}  \norm{\nabla \varphi^{n+1}_h}
	\\
	&\leq \frac{1}{8(1 + \omega_{n}^{2})\dt_{n}} \norm{\varphi_{h}^{n+1} - E^{n+1}(\varphi_{h})}^{2} 
	\\
	&+  \frac{C_{stab} \dt_n(1+\omega_n^2)}{2h} \norm{\nabla E^{n+1}(u_h)}^{2} \norm{\nabla \varphi^{n+1}_h}^{2}
	\\
	&\leq \frac{1}{4\dt_{n}(1+\omega_n^2)}\| \varphi^{n+1}_{h} - \varphi_{h}^{n}\|^{2} + \frac{\omega_n^2}{4\dt_{n}(1+\omega_n^2)}\|\varphi_{h}^{n} - \varphi_{h}^{n-1}\|^{2}  
	\\
	&+ \frac{C_{stab} \dt_n(1+\omega_n^2)}{2h} \norm{\nabla E^{n+1}(u_h)}^{2} \norm{\nabla \varphi^{n+1}_h}^{2}
	\\
	&\leq \frac{1}{4\dt_{n}}\| \varphi^{n+1}_{h} - \varphi_{h}^{n}\|^{2} + \frac{1}{4\dt_{n}}\|\varphi_{h}^{n} - \varphi_{h}^{n-1}\|^{2} 
	\\
	&+  \frac{C_{stab} \dt_n(1+\omega_n^2)}{2h} \norm{\nabla E^{n+1}(u_h)}^{2} \norm{\nabla \varphi^{n+1}_h}^{2}.
	\end{aligned}
	\end{equation}
	Taking $C_{1}, C_{2}, C_{4}, C_{6}, C_{7}  = \frac{1}{64}$, $C_{3},C_{5} = \frac{1}{16}$ the result follows.	
\end{proof}

For the error analysis we will need the following discrete version of Gronwall's inequality found in \cite{B98}.
\begin{lemma}
	Assume that the sequence $\{w_{n}\}$ satisfies
	\begin{equation*}
	w_{n} + c_{n} \leq a_{n} + \sum_{k=0}^{n-1}b_{k}w_{k}, \ \ \ n = 1,2\ldots, N + 1,
	\end{equation*}
	where $\{a_{n}\}$ is nondecreasing and $b_{n}, c_{n} \geq 0$. Then we have the following bound
	\begin{equation*}
	w_{n} + c_{n} \leq a_{n}\exp\left(\sum_{k=0}^{n-1}b_{k}\right).
	\end{equation*} 
\end{lemma}

\begin{theorem}[Error Analysis]
	Consider the VSS-BE-AB2 algorithm \eqref{eqn:VSS-BE-AB2}.  Suppose for any $1 \leq n \leq N-1$, the stability conditions from Theorem \ref{theorem:stability_BE} and the regularity of the solution given in Assumption \ref{assumption:reg} holds. Define the maximum stepsize ratio for $1 \leq n \leq N-1$ as
	\begin{equation*}
	\omega_{N^{*}} = \max_{n = 1, \ldots, N-1} \omega_{n}.
	\end{equation*}
	  We then have the following error estimate
	 \begin{gather*}
	\norm{e_{u}^{N}}^2 + \frac{\nu}{4} \sum_{n=0}^{N-1} \Delta t_{n}\|\nabla e^{n+1}_{u} \|
	\\
	\leq C \biggr[h^{2s + 2} + \nu T h^{2s} + \exp\left(\frac{CT}{\nu^{3}}\right)\biggr\{ \frac{1}{2}\norm{e_{u}^{1}-e_{u}^{0}}^2 
	+ \frac{1}{4}\dt_{1}\norm{\Grad{e_{u}^{1}}}^2 + \frac{1}{8}\dt_{1}\norm{\Grad{e_{u}^{0}}}^2 
	\\
 + \sum_{n=1}^{N-1}\bigg({C h^{2s} \dt_{n}}\norm{\nabla u^{n+1}}^{2} + \frac{C\dt_{n} (\Delta t_{n-1}+ \Delta t_{n})^3}{\nu} \| \nabla u_{tt}\|^{2}_{L^{2}(t^{n-1},t^{n+1} L^{2}(\Omega))}  
\\
+ \frac{Ch\dt_{n}(\Delta t_{n-1}+ \Delta t_{n})^3(1 + h^{2s})}{C_{stab} \dt_n(1+\omega_n^2)} \| \nabla u_{tt}\|^{2}_{L^{2}(t^{n-1},t^{n+1} L^{2}(\Omega))}
\\
+ \frac{C h^{2s}\dt_{n}\omega_{n}}{\nu}\left(\norm{\nabla u^{n}}^{2} + \norm{\nabla u^{n-1}}^{2}\right) + \frac{C\dt_{n}^2}{\nu}\norm[L^2(t^{n},t^{n+1};L^2(\Omega))]{u_{tt}}^2 
\\ 
+ \frac{Ch^{2s}\omega_{N^{*}}}{\nu^{\frac{3}{2}}}\left(\sum_{n=1}^{N-1} \frac{\dt_n}{\nu}\|f^{n+1}\|^{2}_{-1} +
\frac{1}{2}\|u^{1}_{h}\|^{2} + \frac{1}{4}\|u_{h}^{1} - u_{h}^{0}\|^{2}\right) \left(\sum_{n=0}^{N-1}\dt_{n}\norm{\nabla u^{n+1}}^{4}\right)^{\frac{1}{2}} 
\\
+ \frac{C h^{2s}  \dt_{n} (\Delta t_{n-1}+ \Delta t_{n})^3}{\nu} \| \nabla u_{tt}\|^{2}_{L^{2}(t^{n-1},t^{n+1} L^{2}(\Omega))} 
+ \frac{C h^{2s}\dt_{n}}{\nu}\norm{p^{n+1}}^{2} \bigg) \biggr\} \biggr].
	\end{gather*}
\end{theorem}

\begin{proof}
	The true solutions of the NSE satisfies, for all $n=1,\dots,N-1$,
	\begin{equation}\label{eqn:true_NSE}
	\begin{aligned}
	&\left(\frac{u^{n+1}-u^{n}}{\dt_n},v_h\right)+b^*\left(u^{n+1},u^{n+1},v_h\right)+\nu\left(\Grad{u^{n+1}},\Grad{v_h}\right)\\
	&-\left({p}^{n+1},\Div{v_h}\right)=(f^{n+1},v_h)+\tau_u(u^{n+1};v_h),
	\end{aligned}
	\end{equation}
	where $\tau_u(u^{n+1};v_h)$ is defined as
	\begin{equation*}
	\tau_u(u^{n+1};v_h) = \left(\frac{u^{n+1} - u^{n}}{\dt_{n}} - u_{t}(t_{n}),v_{h}\right).
	\end{equation*}
	Subtracting \eqref{eqn:VSS-BE-AB2-LBBh} from \eqref{eqn:true_NSE} yields the error equation
	\begin{equation}
	\begin{aligned}
	&\left(\frac{e_u^{n+1}-e_u^{n}}{\dt_{n}},v_h\right)+b^*\left(u^{n+1},u^{n+1},v_h\right)-b^{\ast}(E^{n+1}(u_h),E^{n+1}(u_h),v_h)\\
	&+\nu\left(\Grad{e_u^{n+1}},\Grad{v_h}\right)-\left(p^{n+1},\Div{v_h}\right)=\tau_u(u^{n+1};v_h).
	\end{aligned}
	\end{equation}
	This can equivalently be written as 
	\begin{gather}
	\left(\frac{\varphi_{h}^{n+1}-\varphi_{h}^{n}}{\dt_{n}},v_h\right)+\nu\left(\Grad{\varphi_{h}^{n+1}},\Grad{v_h}\right)\\
	\notag = \left(\frac{\eta^{n+1}-\eta^{n}}{\dt_{n}},v_h\right)+\nu\left(\Grad{\eta^{n+1}},\Grad{v_h}\right) - (p^{n+1},\nabla \cdot v_h)  \\
	\notag b^*\left(u^{n+1},u^{n+1},v_h\right)-b^{\ast}(E^{n+1}(u_h),E^{n+1}(u_h),v_h)-\tau_u(u^{n+1};v_h).
	\end{gather}
	Letting $v_h = 2\dt_n \varphi_{h}^{n+1}$, using the fact that $2(\eta^{n+1}-\eta^{n},v_h) =0$ by the definition of the $L^{2}$ projection, and the polarization identity yields
	\begin{gather*}
	\norm{\varphi_{h}^{n+1}}^2 - \norm{\varphi_{h}^{n}}^2+\norm{\varphi_{h}^{n+1}-\varphi_{h}^{n}}^2 + 2\dt_{n}\nu\norm{\Grad{\varphi_{h}^{n+1}}}^2 \\
	= 2\dt_{n}\nu(\nabla \eta^{n+1},\nabla \varphi_{h}^{n+1}) - 2 \dt_n (p^{n+1},\nabla \cdot \varphi_{h}^{n+1}) 
	\\
	+ 2\dt_n b^*\left(u^{n+1},u^{n+1},\varphi_{h}^{n+1}\right)
	- 2\dt_n b^{\ast}(E^{n+1}(u_h),E^{n+1}(u_h),\varphi^{n+1}_h) 
	\\
	-2\dt_{n}\tau_u(u^{n+1};\varphi_{h}^{n+1}).
	\end{gather*}
	By the Cauchy-Schwarz-Young and Poincar\'{e}-Friedrichs inequalities, we bound the first term on the right-hand-side
	\begin{gather*}
	2\dt_{n}\nu(\nabla \eta^{n+1},\nabla \varphi_{h}^{n+1})\leq \frac{\nu\dt_{n}}{\delta_{1}}\norm{\nabla \eta^{n+1}}^2+\dt_{n}\delta_{1}\nu\norm{\Grad{\varphi_{h}^{n+1}}}^2.
	\end{gather*}
	Next, we consider the pressure term. Since $\varphi_{h}^{n+1} \in V_{h}$ we have
	\begin{equation*}
	\begin{aligned}
	2 \dt_n (p^{n+1},\nabla \cdot \varphi_{h}^{n+1}) &=  2 \dt_n (p^{n+1} - q^{n+1}_{h},\nabla \cdot \varphi_{h}^{n+1})
	\\ &\leq \frac{\dt_{n}}{\delta_{2}\nu}\norm{p^{n+1} - q^{n+1}_{h}}^{2} + \dt_{n}\delta_{2}\nu\|\nabla \varphi_{h}^{n+1}\|^{2}.
	\end{aligned}
	\end{equation*}
	Using Lemma \ref{lemma:consitency_error} and Cauchy-Schwarz-Young the consistency term is bounded as
	\begin{gather*}
	-2\dt_{n}\tau_u(u^{n+1};\varphi_{h}^{n+1}) \leq \frac{C\dt_{n}^2}{\delta_{3}\nu}\norm[L^2(t^{n-1},t^{n+1};L^2(\Omega))]{u_{tt}}^2 +\dt_{n}\nu \delta_{3}\norm{\Grad{\varphi^{n+1}_h}}^2.
	\end{gather*} 
	Lastly, the nonlinear terms are bounded using Lemma \ref{lemma:nonlinear}
	\begin{equation*}
	\begin{aligned}
	&2\dt_n b^*\left(u^{n+1},u^{n+1},\varphi_{h}^{n+1}\right)
	- 2\dt_n b^{\ast}(E^{n+1}(u_h),E^{n+1}(u_h),\varphi^{n+1}_h) 
	\\
	&\leq 2\dt_{n} \bigg\{\frac{5}{64}\nu\norm{\nabla \varphi_{h}^{n+1}}^{2} + \frac{C}{\nu}\norm{\nabla E^{n+1}(\eta)}^{2} + C \nu^{-3} \norm{\varphi_{h}^{n}}^{2} + C \nu^{-3} \norm{\varphi_{h}^{n-1}}^{2}  
	\\
	& + \frac{C (\Delta t_{n-1}+ \Delta t_{n})^3}{\nu} \| \nabla u_{tt}\|^{2}_{L^{2}(t^{n-1},t^{n+1} L^{2}(\Omega))} + \frac{1}{16}\nu \norm{\nabla \varphi_{h}^{n}}^{2} + \frac{1}{16}\nu \norm{\nabla \varphi_{h}^{n-1}}^{2} 
	\\
	& + \frac{C}{\nu}\norm{E^{n+1}(u_h)}\norm{\nabla E^{n+1}(u_h)}\norm{\nabla \eta^{n+1}}^{2}  + \frac{1}{4\dt_{n}}\| \varphi^{n+1}_{h} - \varphi_{h}^{n}\|^{2}
	\\
	&+ \frac{1}{4\dt_{n}}\|\varphi_{h}^{n} - \varphi_{h}^{n-1}\|^{2} +\frac{5C_{stab} \dt_n(1+\omega_n^2)}{8h} \norm{\nabla E^{n+1}(u_h)}^{2} \norm{\nabla \varphi^{n+1}_h}^{2} 
	\\
	&+ \frac{Ch(\Delta t_{n-1}+ \Delta t_{n})^3}{C_{stab} \dt_n(1+\omega_n^2)} \left(\| \nabla u_{tt}\|^{2}_{L^{2}(t^{n-1},t^{n+1} L^{2}(\Omega))} + \| \nabla \eta_{tt}\|^{2}_{L^{2}(t^{n-1},t^{n+1} L^{2}(\Omega))}\right)\bigg\}.
	\end{aligned}
	\end{equation*}
	Taking $\delta_{1},\delta_{2},\delta_{3} = \frac{1}{32}$, adding and subtracting $\frac{1}{8}\nu\|\nabla \varphi^{n}_{h}\|$, and rearranging/combining terms we have
	\begin{gather*}
	\norm{\varphi_{h}^{n+1}}^2 - \norm{\varphi_{h}^{n}}^2+\frac{1}{2}\norm{\varphi_{h}^{n+1}-\varphi_{h}^{n}}^2 - \frac{1}{2}\norm{\varphi_{h}^{n}-\varphi_{h}^{n-1}}^2 + \frac{\dt_{n}\nu}{4}\|\nabla \varphi^{n+1}_{h}\|
	\\
	+ \frac{5\dt_{n}\nu}{4}\norm{\Grad{\varphi_{h}^{n+1}}}^2 \left(1 - \frac{C_{stab} \dt_n^{2}(1+\omega_n^2)}{ h}\norm{\nabla E^{n+1}(u_h)}^{2}\right)
	\\
	+ \frac{\dt_n}{4}\left(\norm{\Grad{\varphi_{h}^{n+1}}}^2 - \norm{\Grad{\varphi_{h}^{n}}}^2 \right) + \frac{\dt_n}{8}\left(\norm{\Grad{\varphi_{h}^{n}}}^2 - \norm{\Grad{\varphi_{h}^{n-1}}}^2 \right)
	\\
	\leq   \frac{C\dt_{n}}{\nu}\norm{E^{n+1}(u_h)}\norm{\nabla E^{n+1}(u_h)}\norm{\nabla \eta^{n+1}}^{2} + \frac{C\dt_{n}}{\nu}\norm{\nabla E^{n+1}(\eta)}^{2}
	\\
	+ \frac{C\dt_{n} (\Delta t_{n-1}+ \Delta t_{n})^3}{\nu} \| \nabla u_{tt}\|^{2}_{L^{2}(t^{n-1},t^{n+1} L^{2}(\Omega))} + C\nu\Delta t_{n}\|\nabla \eta ^{n+1}\|^{2}
	\\
	 + \frac{Ch \dt_{n}(\Delta t_{n-1}+ \Delta t_{n})^3}{C_{stab} \dt_n(1+\omega_n^2)} \left(\| \nabla u_{tt}\|^{2}_{L^{2}(t^{n-1},t^{n+1} L^{2}(\Omega))} + \| \nabla \eta_{tt}\|^{2}_{L^{2}(t^{n-1},t^{n+1} L^{2}(\Omega))}\right)
	\\
	+ C\dt_{n} \nu^{-3} \norm{\varphi_{h}^{n}}^{2} + C \dt_{n} \nu^{-3} \norm{\varphi_{h}^{n-1}}^{2} 
	+ \frac{C\dt_{n}^2}{\nu}\norm[L^2(t^{n-1},t^{n+1};L^2(\Omega))]{u_{tt}}^2 
	\\
	+ \frac{C \dt_{n}}{\nu}\norm{p^{n+1} - q^{n+1}_{h}}^{2}.
	\end{gather*}
	We note that using Cauchy-Schwarz-Young and the stability estimate from Theorem \ref{theorem:stability_BE} we have that 
	\begin{equation*}
	\begin{aligned}
	&\sum_{n=0}^{N-1}\frac{C \dt_{n}}{\nu}\norm{E^{n+1}(u_h)}\norm{\nabla E^{n+1}(u_h)}\norm{\nabla \eta^{n+1}}^{2} 
	\\
	&\leq \frac{C}{\nu}\left(\max_{n = 0, \ldots, N-1}\norm{E^{n+1}(u_h)}\right)\sum_{n=0}^{N-1}\dt_{n}\norm{\nabla E^{n+1}(u_h)}\norm{\nabla \eta^{n+1}}^{2} 
	\\
	&\leq \frac{C\omega_{N^{*}}^{\frac{1}{2}}}{\nu}\left(\sum_{n=1}^{N-1} \frac{\dt_n}{\nu}\|f^{n+1}\|^{2}_{-1} +
 \frac{1}{2}\|u^{1}_{h}\|^{2} + \frac{1}{4}\|u_{h}^{1} - u_{h}^{0}\|^{2}\right)^{\frac{1}{2}} \times
 \\
  &\sum_{n=0}^{N-1}\dt_{n}\norm{\nabla E^{n+1}(u_h)}\norm{\nabla \eta^{n+1}}^{2} 
  \\
	&\leq \frac{C\omega_{N^{*}}^{\frac{1}{2}}}{\nu}\left(\sum_{n=1}^{N-1} \frac{\dt_n}{\nu}\|f^{n+1}\|^{2}_{-1} +
\frac{1}{2}\|u^{1}_{h}\|^{2} + \frac{1}{4}\|u_{h}^{1} - u_{h}^{0}\|^{2}\right)^{\frac{1}{2}} \times
\\
&\left(\sum_{n=0}^{N-1}\dt_{n}\norm{\nabla E^{n+1}(u_h)}^{2}\right)^{\frac{1}{2}}\left(\sum_{n=0}^{N-1}\dt_{n}\norm{\nabla \eta^{n+1}}^{4}\right)^{\frac{1}{2}}
\\
&\leq \frac{C\omega_{N^{*}}}{\nu^{\frac{3}{2}}}\left(\sum_{n=1}^{N-1} \frac{\dt_n}{\nu}\|f^{n+1}\|^{2}_{-1} +
\frac{1}{2}\|u^{1}_{h}\|^{2} + \frac{1}{4}\|u_{h}^{1} - u_{h}^{0}\|^{2}\right) \left(\sum_{n=0}^{N-1}\dt_{n}\norm{\nabla \eta^{n+1}}^{4}\right)^{\frac{1}{2}}.
	\end{aligned}
	\end{equation*}
Then, using Theorem \ref{theorem:stability_BE}, summing from $n=1$ to $n = N -1$, dropping positive terms on the left hand side, and using the above bound we have 
	\begin{gather*}
	\norm{\varphi_{h}^{N}}^2 + \frac{\nu}{4}\sum_{n=0}^{N-1} \dt_{n}\|\nabla \varphi^{n+1}_{h}\| \leq  \frac{1}{2}\norm{\varphi_{h}^{1}-\varphi_{h}^{0}}^2 
	+ \frac{1}{4}\dt_{1}\norm{\Grad{\varphi_{h}^{1}}}^2 + \frac{1}{8}\dt_{1}\norm{\Grad{\varphi_{h}^{0}}}^2  
	 \\
	 + C \nu^{-3} \sum_{n=0}^{N-1} \dt_{n} \norm{\varphi_{h}^{n}}^{2}
	+ \sum_{n=1}^{N-1}\bigg\{{C\dt_{n}\nu}\norm{\nabla \eta^{n+1}}^{2} + \frac{C\Delta t_{n}\omega_{n}}{\nu}\left( \|\nabla \eta^{n}\|^{2} + \|\nabla \eta^{n-1}\|^{2}\right)
	\\
	 + \frac{Ch \dt_{n}(\Delta t_{n-1}+ \Delta t_{n})^3}{C_{stab} \dt_n(1+\omega_n^2)} \left(\| \nabla u_{tt}\|^{2}_{L^{2}(t^{n-1},t^{n+1} L^{2}(\Omega))} + \| \nabla \eta_{tt}\|^{2}_{L^{2}(t^{n-1},t^{n+1} L^{2}(\Omega))}\right)
	\\
	+ \frac{C\dt_{n} (\Delta t_{n-1}+ \Delta t_{n})^3}{\nu} \| \nabla u_{tt}\|^{2}_{L^{2}(t^{n-1},t^{n+1} L^{2}(\Omega))}  
  + \frac{C\dt_{n}^2}{\nu}\norm[L^2(t^{n},t^{n+1};L^2(\Omega))]{u_{tt}}^2 
  \\
	+ \frac{C\omega_{N^{*}}}{\nu^{\frac{3}{2}}}\left(\sum_{n=1}^{N-1} \frac{\dt_n}{\nu}\|f^{n+1}\|^{2}_{-1} +
	\frac{1}{2}\|u^{1}_{h}\|^{2} + \frac{1}{4}\|u_{h}^{1} - u_{h}^{0}\|^{2}\right) \left(\sum_{n=0}^{N-1}\dt_{n}\norm{\nabla \eta^{n+1}}^{4}\right)^{\frac{1}{2}} 
	\\ + \frac{C \dt_{n}}{\nu}\norm{p^{n+1} - q^{n+1}_{h}}^{2}\bigg\}.
	\end{gather*}
 Next, invoking the discrete Gronwall's inequality and applying interpolation inequalities gives
 \begin{gather*}
 \norm{\varphi_{h}^{N}}^2 + \frac{\nu}{4}\sum_{n=0}^{N-1} \dt_{n}\|\nabla \varphi^{n+1}_{h}\| \leq 
 \\
 C \exp\left(\frac{CT}{\nu^{3}}\right)\biggr\{ \frac{1}{2}\norm{\varphi_{h}^{1}-\varphi_{h}^{0}}^2 
 + \frac{1}{4}\dt_{1}\norm{\Grad{\varphi_{h}^{1}}}^2 + \frac{1}{8}\dt_{1}\norm{\Grad{\varphi_{h}^{0}}}^2 
 \\ 
  + \sum_{n=1}^{N-1}\bigg({C h^{2s} \dt_{n}}\norm{\nabla u^{n+1}}^{2} + \frac{C\dt_{n} (\Delta t_{n-1}+ \Delta t_{n})^3}{\nu} \| \nabla u_{tt}\|^{2}_{L^{2}(t^{n-1},t^{n+1} L^{2}(\Omega))}  
 \\
 + \frac{Ch\dt_{n}(\Delta t_{n-1}+ \Delta t_{n})^3(1 + h^{2s})}{C_{stab} \dt_n(1+\omega_n^2)} \| \nabla u_{tt}\|^{2}_{L^{2}(t^{n-1},t^{n+1} L^{2}(\Omega))}
 \\
 + \frac{C h^{2s}\dt_{n}\omega_{n}}{\nu}\left(\norm{\nabla u^{n}}^{2} + \norm{\nabla u^{n-1}}^{2}\right) + \frac{C\dt_{n}^2}{\nu}\norm[L^2(t^{n},t^{n+1};L^2(\Omega))]{u_{tt}}^2 
  \\ 
 	+ \frac{Ch^{2s}\omega_{N^{*}}}{\nu^{\frac{3}{2}}}\left(\sum_{n=1}^{N-1} \frac{\dt_n}{\nu}\|f^{n+1}\|^{2}_{-1} +
 \frac{1}{2}\|u^{1}_{h}\|^{2} + \frac{1}{4}\|u_{h}^{1} - u_{h}^{0}\|^{2}\right) \left(\sum_{n=0}^{N-1}\dt_{n}\norm{\nabla u^{n+1}}^{4}\right)^{\frac{1}{2}} 
 \\
 + \frac{C h^{2s}  \dt_{n} (\Delta t_{n-1}+ \Delta t_{n})^3}{\nu} \| \nabla u_{tt}\|^{2}_{L^{2}(t^{n-1},t^{n+1} L^{2}(\Omega))} 
 + \frac{C h^{2s}\dt_{n}}{\nu}\norm{p^{n+1}}^{2} \bigg) \biggr\}.
 \end{gather*}
 	Finally, by the triangle inequality we have $ e^{n}_{u} \leq 2(\varphi^n_h + \eta^{n})$. Applying this inequality, interpolation inequalities, and absorbing constants, the result follows. 
\end{proof}
\section{The second order method\label{sec:second_order}}
We now describe and analyze the second order member in the VSVO method in Section \ref{sec:vsvo}, which is 

\begin{algorithm}\label{algo:VSS-filtere-BE-AB2}[VSS Filtered-BE-AB2]
	Given $\Delta t$, $(u_{h}^{n},p_{h}^{n})$, \newline $(u_{h}^{n-1},p_{h}^{n-1})$, find $(\hat{u}_{h}^{n+1},p_{h}^{n+1})$ satisfying
	\begin{equation}
	\begin{aligned}
	\left(\frac{\hat{u}_{h}^{n+1} - u_{h}^{n}}{\dt},v_h\right) + \nu(\nabla \hat{u}^{n+1}_h,\nabla v_h) + &b^{\ast}(E^{n+1}(u_h),E^{n+1}(u_h),v_h)  
	\\ - (p_{h}^{n+1},\nabla \cdot v_h)  &= (f^{n+1},v_{h}) \qquad \qquad \forall
	v_{h}\in X_{h}
	\\
	(\nabla \cdot \hat{u}_{h}^{n+1},q_{h}) &= 0 \qquad \qquad \qquad \qquad \forall
	q_{h}\in Q_{h}.
	\end{aligned}
	\end{equation}
	Then, compute
	\begin{equation}
	u_{h}^{n+1} = \hat{u}_{h}^{n+1} - \frac{\omega_{n}}{2\omega_{n}+1}(\hat{u}_{h}^{n+1} - E^{n+1}(u_h)).
	\end{equation}
\end{algorithm}

The first step in the method is exactly BE-AB2, but the temporary solution $\hat{u}_h^{n+1}$ is replaced by a corrected solution, which shows the embedded structure. This method is formally second order. Indeed, by eliminating the intermediate variable $\hat{u}_h^{n+1}$, one can show that the method is a second order perturbation of variable stepsize backward differentiation formula 2 (VSS-BDF2).

While proving energy ability for the variable stepsize method would be a major contribution, we do not yet have a proof (proving energy stability of VSS-BDF2 is already challenging and has only recently been proven for the Cahn-Hilliard equations in \cite{WXYZ19}). For now, we prove stability for the constant stepsize version, although the numerical tests in Section \ref{sec:tests} indicate that VSS Filtered-BE-AB2 is still stable. For brevity, we do not include a full convergence analysis of the method.

\begin{algorithm}\label{algo:CSS-filtere-BE-AB2}[Constant Timestep Filtered-BE-AB2 (BE-AB2+F)]
	Given $\Delta t$, $(u_{h}^{n},p_{h}^{n})$, \newline $(u_{h}^{n-1},p_{h}^{n-1})$, find $(\hat{u}_{h}^{n+1},p_{h}^{n+1})$ satisfying
	\begin{equation}
	\begin{aligned}
	\left(\frac{\hat{u}_{h}^{n+1} - u_{h}^{n}}{\dt},v_h\right) + \nu(\nabla \hat{u}^{n+1}_h,\nabla v_h) + &b^{\ast}(2u_{h}^{n}-u_{h}^{n-1},2u_{h}^{n}-u_{h}^{n-1},v_h)  
	\\ - (p_{h}^{n+1},\nabla \cdot v_h)  &= (f^{n+1},v_{h}) \qquad \qquad \forall
	v_{h}\in X_{h}
	\\
	(\nabla \cdot \hat{u}_{h}^{n+1},q_{h}) &= 0 \qquad \qquad \qquad \qquad \forall
	q_{h}\in Q_{h}.
	\end{aligned}
	\end{equation}
	Then, compute
	\begin{equation}
	u_{h}^{n+1} = \hat{u}_{h}^{n+1} - \frac{1}{3}(\hat{u}^{n+1}_{h} - 2u_{h}^{n} + u_{h}^{n-1}).
	\end{equation}
	Equivalently, this can be written as
	\begin{equation}\label{coupled-Filtered-BE-AB2}
	\begin{aligned}
	\left(\frac{\frac{3}{2}{u}_{h}^{n+1} - 2u_{h}^{n} + \frac{1}{2}u_{h}^{n-1}}{\dt},v_h\right) + \nu\left(\nabla \left(\frac{3}{2}{u}_{h}^{n+1} - u_{h}^{n} + \frac{1}{2}u_{h}^{n-1}\right),\nabla v_h\right)&  \\
	+ b^{\ast}(2u_{h}^{n}-u_{h}^{n-1},2u_{h}^{n}-u_{h}^{n-1},v_h) - (p^{n+1},\nabla \cdot v_h)&  = (f^{n+1},v_{h}) 
	\\
	(\nabla \cdot \hat{u}_{h}^{n+1},q_{h}) &= 0. 
	\end{aligned}
	\end{equation}
\end{algorithm}


In order to prove stability we will need to use the identity

\begin{lemma}\label{identity_2}
	The following identity holds
	\begin{gather*}
	\left(\frac{3}{2}a - 2b + \frac{1}{2}c,\frac{3}{2}a - b + \frac{1}{2}c\right) = \\
	\left(\frac{\|a\|^{2}}{4} + \frac{\|2a -b \|^{2}}{4} + \frac{\|a-b\|^{2}}{4}\right) - \left(\frac{\|b\|^{2}}{4} + \frac{\|2b -c \|^{2}}{4} + \frac{\|b-c\|^{2}}{4}\right) \\
	+ \frac{3}{4}\|a - 2b + c\|^{2}.
	\end{gather*}
\end{lemma}

We then have the following general conditional stability result. This result can be improved further using the same techniques as those in Section \ref{section:stab}. Stability and convergence of the VSS version of this method is currently an open problem. 
\begin{theorem}
Consider the method \eqref{coupled-Filtered-BE-AB2}, let $\Omega \subset \mathbb{R}^{d}, d =2,3,$ and $C_{stab} > 0$ be a constant independent of $h,\Delta t, \nu$ and $u$. Suppose that
\begin{equation}\label{stab_cond_2}
1 - \frac{C_{stab} \dt}{\nu h}\| \nabla (2u_{h}^{n}-u_{h}^{n-1})\|^{2} \geq 0.
\end{equation}
Then, for any $N> 1$
\begin{gather*}
\frac{1}{4}\|u^{N}_{h}\|^{2} + \frac{1}{4}\| 2u_{h}^{N} -  u_{h}^{N-1}\|^{2} + \frac{1}{4}\|u_{h}^{N} - u_{h}^{N-1} \|^{2}
\\
+\frac{\nu \dt}{4}\sum_{n=1}^{N-1}\|\nabla (\frac{3}{2}{u}_{h}^{n+1} - u_{h}^{n} + \frac{1}{2}u_{h}^{n-1})\|^{2}
 \\ \leq \sum_{n=1}^{N-1} \frac{\dt}{\nu}\|f^{n+1}\|^{2}_{-1} + \frac{1}{4}\|u^{1}_{h}\|^{2} + \frac{1}{4}\|2u_{h}^{1} - u_{h}^{0}\|^{2} + \frac{1}{4}\|u_{h}^{1} - u_{h}^{0} \|^{2}.
\end{gather*}
\end{theorem}
\begin{proof}
Setting $v_h = \frac{3}{2}{u}_{h}^{n+1} - u_{h}^{n} + \frac{1}{2}u_{h}^{n-1}$, multiplying by $\dt$, using Lemma \ref{identity_2}, and applying Young's inequality to the right hand side
\begin{gather*}
\frac{1}{4}(\|u^{n+1}_{h}\|^{2} + \| 2u_{h}^{n+1} - u_{h}^{n}\|^{2} + \|u_{h}^{n+1} - u_{h}^{n} \|^{2}) - 
\\
\frac{1}{4}(\|u_{h}^{n}\|^{2} + \|2u_{h}^{n} - u_{h}^{n-1}\|^{2} + \|u_{h}^{n} - u_{h}^{n-1}\|^{2} ) +
\\
\frac{3}{4}\|u_{h}^{n+1} - 2u_{h}^{n} + u_{h}^{n-1}\|^{2} + \dt\nu\|\nabla (\frac{3}{2}{u}_{h}^{n+1} - u_{h}^{n} + \frac{1}{2}u_{h}^{n-1})\|^{2}+
\\
\dt b^{\ast}(2u_{h}^{n}-u_{h}^{n-1},2u_{h}^{n}-u_{h}^{n-1},\frac{3}{2}{u}_{h}^{n+1} - u_{h}^{n} + \frac{1}{2}u_{h}^{n-1}) \\  \leq \frac{\nu \dt}{4}\|\nabla(\frac{3}{2}{u}_{h}^{n+1} - 2u_{h}^{n} + \frac{1}{2}u_{h}^{n-1})\|^{2} + \frac{\dt}{\nu}\|f^{n+1}\|_{-1}^{2}.
\end{gather*}
Next, dealing with the nonlinear term we use the skew symmetry of $b^{\ast}$, Poincaré inequality, inequality \eqref{In1}, the inverse inequality and Young's inequality
\begin{equation*}
\begin{aligned}
\dt b^{\ast}&(2u_{h}^{n}-u_{h}^{n-1},2u_{h}^{n}-u_{h}^{n-1},\frac{3}{2}{u}_{h}^{n+1} - u_{h}^{n} + \frac{1}{2}u_{h}^{n-1})
\\
&= -\dt b^{\ast}(2u_{h}^{n}-u_{h}^{n-1},\frac{3}{2}{u}_{h}^{n+1} - u_{h}^{n} +\frac{1}{2}u_{h}^{n-1},-2u_{h}^{n}+u_{h}^{n-1})  
\\
&= -\frac{3}{2}\dt b^{\ast}(2u_{h}^{n}-u_{h}^{n-1},\frac{3}{2}{u}_{h}^{n+1} - u_{h}^{n} +\frac{1}{2}u_{h}^{n-1},{u}_{h}^{n+1} - 2u_{h}^{n} +u_{h}^{n-1})  
\\
&\leq \frac{3}{2} C_{b^{\ast}} \dt \| \nabla (2u_{h}^{n}-u_{h}^{n-1})\| \|\nabla(\frac{3}{2}{u}_{h}^{n+1} - u_{h}^{n} -\frac{1}{2}u_{h}^{n-1}) \|  \\
&\| \nabla ({u}_{h}^{n+1} - 2u_{h}^{n} -u_{h}^{n-1})\|^{1/2} \| ({u}_{h}^{n+1} - 2u_{h}^{n} -u_{h}^{n-1})\|^{1/2} \\
&\leq C \dt h^{-\frac{1}{2}}\| \nabla (2u_{h}^{n}-u_{h}^{n-1})\|\nabla(\frac{3}{2}{u}_{h}^{n+1} - u_{h}^{n} -\frac{1}{2}u_{h}^{n-1}) \| \| ({u}_{h}^{n+1} - 2u_{h}^{n} -u_{h}^{n-1})\|\\
&\leq C \frac{\dt^{2}}{h}\| \nabla (2u_{h}^{n}-u_{h}^{n-1})\|^{2}\| \nabla (\frac{3}{2}{u}_{h}^{n+1} - u_{h}^{n} -\frac{1}{2}u_{h}^{n-1})\|^{2} 
\\
&+ \frac{3}{4}\| u_{h}^{n+1}-2u_{h}^{n}+u_{h}^{n-1} \|^{2}.
\end{aligned}
\end{equation*}
Combining like terms we then have
\begin{gather*}
\frac{1}{4}(\|u^{n+1}_{h}\|^{2} + \| 2u_{h}^{n+1} - u_{h}^{n}\|^{2} + \|u_{h}^{n+1} - u_{h}^{n} \|^{2}) - 
\\
\frac{1}{4}(\|u_{h}^{n}\|^{2} + \|2u_{h}^{n} - u_{h}^{n-1}\|^{2} + \|u_{h}^{n} - u_{h}^{n-1}\|^{2} ) +
\\
 \frac{\dt\nu}{4 }\|\nabla (\frac{3}{2}{u}_{h}^{n+1} - u_{h}^{n} + \frac{1}{2}u_{h}^{n-1})\|^{2}+
\\
\frac{\nu \dt}{2}\left(1 - \frac{C\dt}{\nu h}\|\nabla(2u_{h}^{n} - u_{h}^{n-1})\|^{2}\right)\|\nabla (\frac{3}{2}{u}_{h}^{n+1} - u_{h}^{n} + \frac{1}{2}u_{h}^{n-1})\|^{2}
\\ \leq \frac{\dt}{\nu}\|f^{n+1}\|_{-1}^{2}.
\end{gather*}
Now, letting $C = C_{stab}$, using condition \eqref{stab_cond_2}, and summing from $n = 1$ to $N-1$ the result follows. 
\end{proof}	

\section{The VSVO algorithm\label{sec:vsvo}}
We now combine the methods analyzed in Sections \ref{sec:first_order} and \ref{sec:second_order} into a single VSVO method. Rather than discarding the intermediate first order approximation in Algorithm \ref{algo:VSS-filtere-BE-AB2}, it is kept so that we have two approximations to choose from. The first order method, which is provably energy stable for variable stepsizes, and second order method which has a smaller consistency error, and is at least provably energy stable for constant stepsizes.

\begin{algorithm}[Multiple order, one solve embedded - IMEX - 12 (MOOSE-IMEX-12)\label{alg:vsvo_themethod}]\hfill

\noindent
\begin{equation}
	\begin{aligned}
	\left(\frac{u_{h,1}^{n+1} - u_{h}^{n}}{\dt},v_h\right) + \nu(\nabla u^{n+1}_{h,1},\nabla v_h) + &b^{\ast}(E^{n+1}(u_h),E^{n+1}(u_h),v_h)  
	\\ - (p_{h}^{n+1},\nabla \cdot v_h)  &= (f^{n+1},v_{h}) \qquad \qquad \forall
	v_{h}\in X_{h}
	\\
	(\nabla \cdot u_{h,1}^{n+1},q_{h}) &= 0 \qquad \qquad \qquad \qquad \forall
	q_{h}\in Q_{h}.
	\end{aligned}
	\end{equation}

\noindent
\begin{equation}\notag
u_{h,2}^{n+1} = \hat{u}_{h,1}^{n+1} - \frac{\omega_{n}}{2\omega_{n}+1}(\hat{u}_{h,1}^{n+1} - E^{n+1}(u_h))
\end{equation}

\begin{equation}\notag
EST_1=u_{h,2}^{n+1}-u_{h,1}^{n+1}
\end{equation} 
\begin{gather}\label{eqn:estimator_second_order}
EST_2=\frac{\omega_{n-1} \omega_{n} (1+\omega_{n})}{1+2 \omega_{n}+\omega_{n-1} \left(1+4 \omega_{n}+3 \omega_{n}^2\right)}\bigg(u_{h,2}^{n+1} \notag
\\
-\frac{(1+\omega_{n}) (1+\omega_{n-1} (1+\omega_{n}))}{1+\omega_{n-1}}u_{h}^{n} +\omega_{n} (1+\omega_{n-1} (1+\omega_{n}))u_{h}^{n-1}\notag
 \\
 -\frac{\omega_{n-1}^2 \omega_{n} (1+\omega_{n})}{1+\omega_{n-1}}u_{h}^{n-2}\bigg). \notag
\end{gather}

If $\|EST_1\|<TOL$ or $\|EST_2\|<TOL$, at least one approximation is acceptable. Go to Step 5a. Otherwise, the step is rejected. Go to Case 2. 

\textbf{Case 1 : A solution is accepted.}
\begin{equation}\notag
\Delta t^{(1)} = \gamma\Delta t_n \left(\frac{TOL}{\|EST_1\|}\right)^{\frac{1}{2}}, \hspace{15mm}\Delta t^{(2)}= \gamma\Delta t_n \left(\frac{TOL}{\|EST_2\|}\right)^{\frac{1}{3}}.
\end{equation}

Set $$i=\argmax_{i\in \{1,2\}} \Delta t^{(i)}, \hspace{8mm}\Delta t_{n+1}=\Delta t^{(i)}, \hspace{8mm} t^{n+2}=t^{n+1}+\Delta t_{n+1}, \hspace{ 8mm} u_{h}^{n+1} = u_{h,i}^{n+1} .$$

If only $y^{(1)}$ (resp. $y^{(2)}$) satisfies $TOL$, set $\Delta t_{n+1}=\Delta t^{(1)}$ (resp. $\Delta t^{(2)}$), and $y^{n+1} = y_{(1)}^{n+1}$ (resp. $y_{(2)}^{n+1}$). Proceed to calculate $u_h^{n+2}$.

\noindent
\textbf{Case 2 : Both solutions are rejected.}

Set 
\begin{equation}\notag
\Delta t^{(1)} = \tilde{\gamma}\Delta t_n \left(\frac{TOL}{\|EST_1\|}\right)^{\frac{1}{2}}, \hspace{15mm}\Delta t^{(2)}= \tilde{\gamma}\Delta t_n \left(\frac{TOL}{\|EST_2\|}\right)^{\frac{1}{3}}.
\end{equation}

Set $$i=\argmax_{i\in \{1,2\}} \Delta t^{(i)}, \hspace{10mm}\Delta t_{n}=\Delta t^{(i)}, \hspace{10mm} t^{n+1}=t^{n}+\Delta t_{n}$$

Recalculate $u_{h,1}^{n+1}$ and $u_{h,2}^{n+1}$.

\end{algorithm}

The numbers $\gamma$ and $\tilde{\gamma}$ are heuristic safety factors. $\gamma = 0.9$ is a commonly chosen value for adaptive codes. We use the same choices as the implicit version used in \cite{DLZ18}, which were $\gamma = 0.9$, and $\tilde{\gamma} = 0.7$. Case 2 can optionally be replaced with a simpler heuristic where $\Delta t$ is halved.
 
The error estimator $EST_2$ effectively turns MOOSE-IMEX-12 into a three step method, increasing memory complexity. If low storage is important, an alternate error estimator similar to one used in MOOSE234 in \cite{decaria2018new} is more suitable. It is obtained by solving
\begin{gather*}
(EST_2, v_h) = \frac{1}{\Delta t_n}\left(\frac{1+2\omega_n}{1+\omega_n}u_h^{n+1,2} - (1+\omega_n)u_h^{n} + \frac{\omega_n^2}{1+\omega_n}u_h^{n-1},v_h\right)\\
+ \nu (\nabla u_h^{n+1,2},\nabla v_h)+b^*(u_h^{n+1,2},u_h^{n+1,2},v^h)\\ - (p_h^{n+1},\nabla \cdot v_h) - (f(t^{n+1}),v_h) 
 \hspace{10mm} \forall v_h \in X_h
\end{gather*}

$EST_2$ is the residual of the VS-BE-AB2+F solution plugged into the VSS-BDF2 equation. It only requires a mass matrix solve with an $\mathcal{O}(1)$ condition number, which may be solved efficiently with many iterative methods. This version makes MOOSE-IMEX-12 have the same memory complexity as BE-AB2, with slightly increased floating point operations per step.

\section{Numerical Experiments\label{sec:tests}}

We now test the methods on two different problems with known exact solutions to verify the predicted convergence rates. We first test convergence of the constant stepsize, constant order methods in Section \ref{sec:test_nonadaptive} on the well known 2D Taylor-Green vortex problem. Next, we test both adaptive and nonaptive methods on a modified Taylor-Green problem with periodic, rapid transients in Section \ref{sec:test_nonadaptive}. We demonstrate that the new adaptive methods are more efficient than their nonadaptive counterparts.

We now recall the naming conventions for the various methods that we test. BE-AB2+F is BE-AB2 post-processed by the time filter. MOOSE-IMEX-12 refers to Algorithm \ref{alg:vsvo_themethod}. We specify if a method is constant order, constant stepsize by ``nonadaptive''. VSS BE-AB2 still computes $EST_1$ and $\Delta t^{(1)}$ as in Algorithm \ref{alg:vsvo_themethod}, and always uses the first order solution to advance in time. Another way to view VSS BE-AB2 is as a variant of Algorithm \ref{alg:vsvo_themethod} where $EST_2 := \infty$, so that effectively only the first order solution is considered. VSS BE-AB2+F is defined analogously, and can be seen as Algorithm \ref{alg:vsvo_themethod} with $EST_1 := \infty$, so that only the second order solution is used.

For all adaptive methods, we imposed a stepsize ratio limiter, which is a common heuristic. The stepsizes are limited to at most doubling each timestep, and cannot be less than half of the previous attempted $\Delta t$. However, the algorithms may reject several solutions in a row, effectively allowing $\Delta t$ to shrink as small as necessary.

All errors are calculated in the relative $\ell^2(0,T;L^2(\Omega))$ norm,
\begin{equation}\notag
\|u_h-u\|_{\ell^2(0,T;L^2(\Omega))} = \sqrt{\frac{\sum_n {\Delta t_n}\|u_h(t^{n+1})-u(t^{n+1})\|^2}{\sum_n {\Delta t_n}\|u(t^{n+1})\|^2}}.
\end{equation}

\noindent
All tests were performed with the FEniCS project, \cite{fenics}, and our code is available online\footnote{All code and data are available at https://github.com/vpdecaria/beab2}.
\subsection{Accuracy of the nonadaptive method\label{sec:test_nonadaptive}}
We first test the accuracy of the new constant order, constant stepsize methods BE-AB2 and BE-AB2+F. We also compare these methods with the standard BE-FE method. This is done with the decaying Taylor-Green vortex with $f\equiv 0$. In 2D exact solutions are  known and this problem serves as a standard benchmark problem \cite{B05}. The exact solution is given by
\begin{equation}\notag
u = \exp(-2\nu t)\langle\cos x \sin y, -\sin x \cos y\rangle, \hspace{5mm} p = -\frac{1}{4}\exp(-4\nu t)\langle\cos 2x + \cos 2y\rangle. 
\end{equation}
The domain was taken to be the $2\pi$ periodic square, and was meshed with a standard uniform triangulation with 50 triangle edges per side of the square. The elements used were Taylor-Hood, cubic velocities, and quadratic pressures, which are known to satisfy the discrete inf-sup condition. The problem was run till a final time of $T=1$, with $\nu=1$.

The results shown in Figure \ref{fig:const_step_converge} confirm the predicted convergence rates. Interestingly, BE-FE and BE-AB2 produce nearly identical velocity errors, but BE-AB2 has a much improved pressure error. 

\begin{figure}
\centering
\begin{subfigure}{\textwidth}
\includegraphics[width=.495\linewidth]{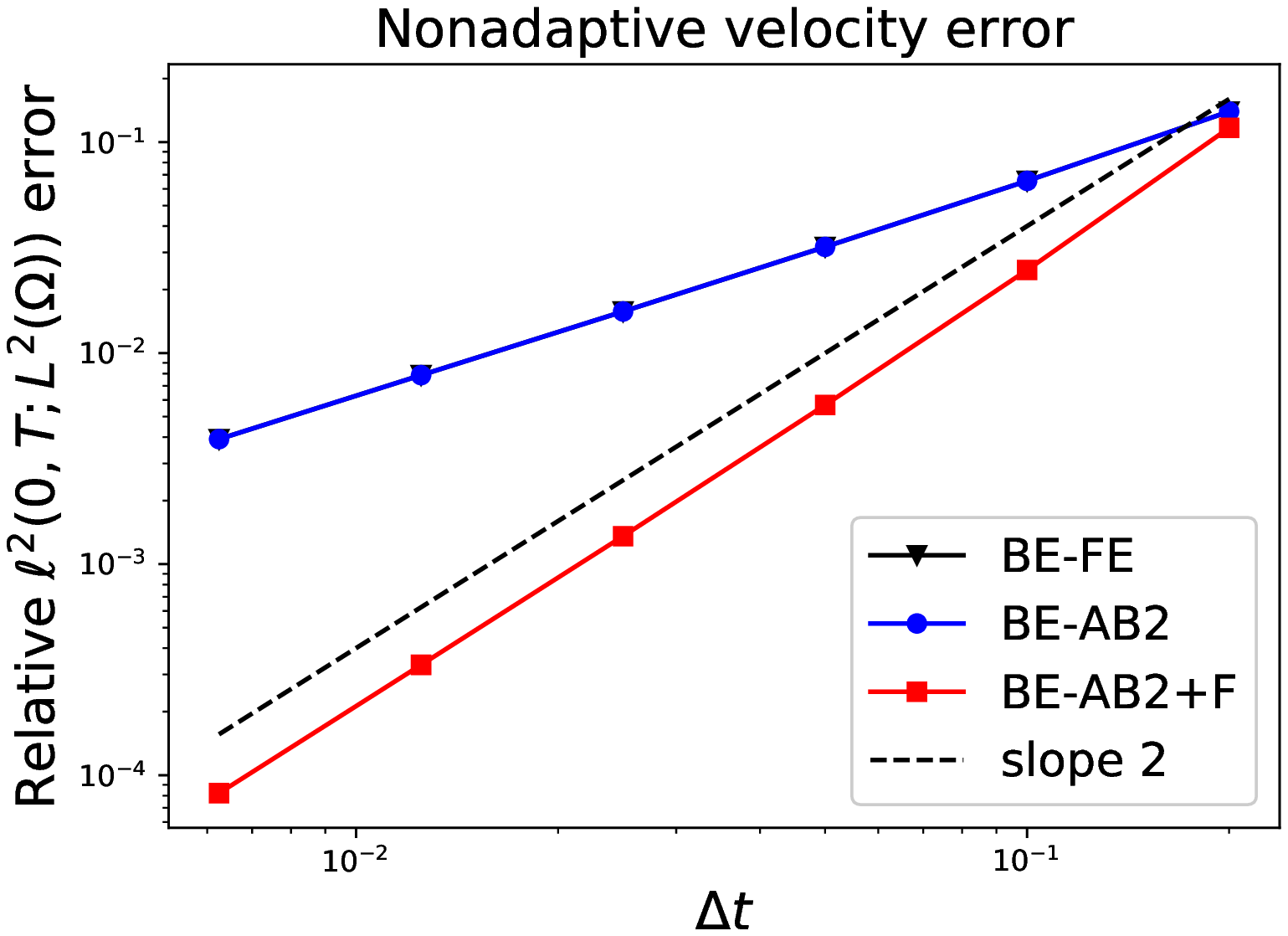}
\includegraphics[width=.495\linewidth]{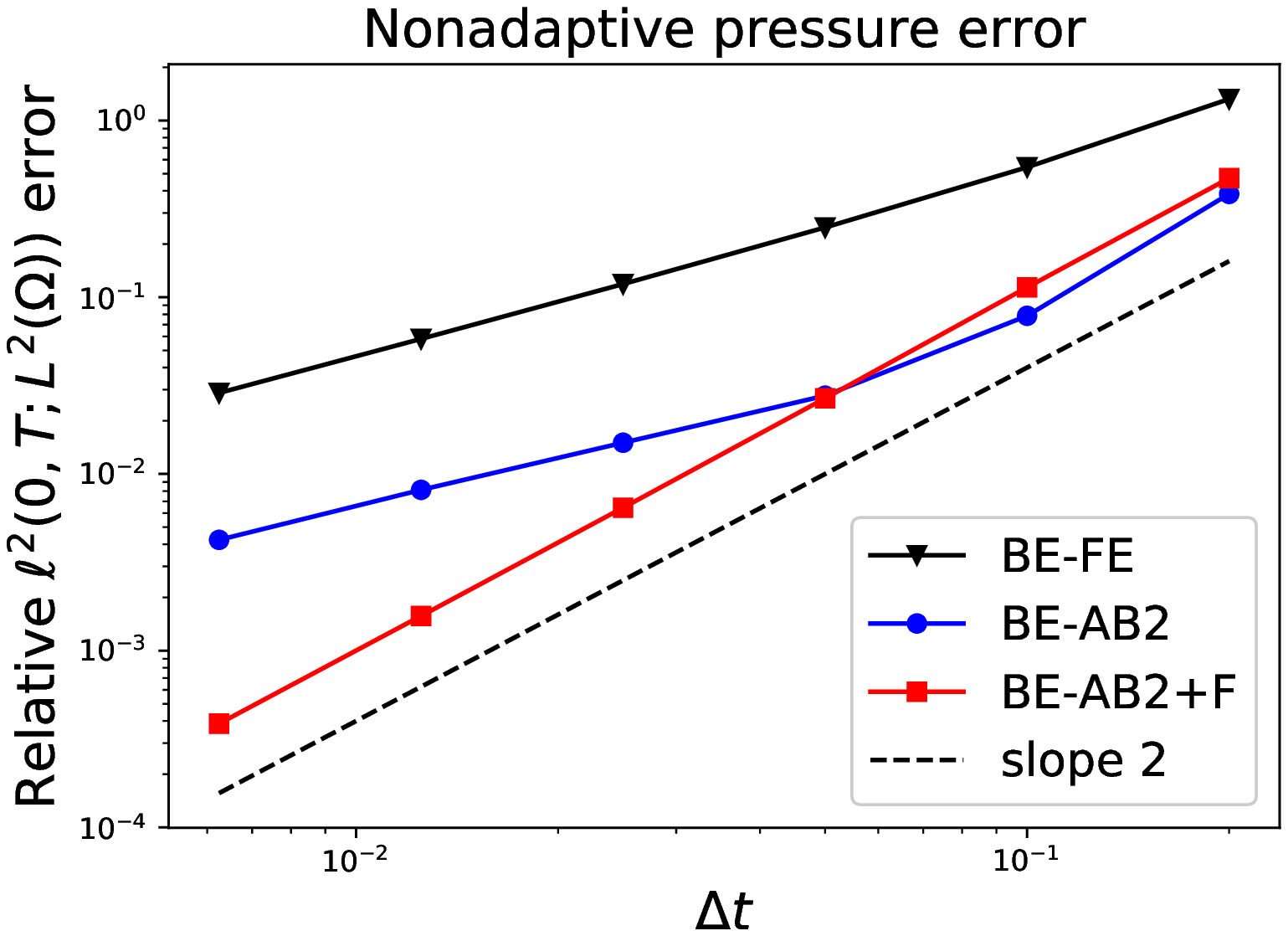}
\end{subfigure}
\caption{The velocity and pressure errors converge at the predicted rates.\label{fig:const_step_converge}}
\end{figure}

\subsection{Accuracy of the fully adaptive, VSVO method\label{sec:test_adaptive}}
Now we test the accuracy and robustness of the adaptive methods on a problem with a fast and slow time scale, which demonstrates the superiority of the adaptive methods in this case. This is a modification of the Taylor-Green vortex problem with a nonautonomous body force that causes periodic, rapid transients. This test was performed in \cite{DLZ18}, and is described here for completeness.

Let $F = F(t)$ be differentiable. For the following body force,
\begin{equation}\notag
f(x,y,t) = (2\nu F(t) + F'(t))\langle \cos x \sin y, -\cos y \sin x \rangle,
\end{equation}
an exact solution is given by
\begin{equation}\notag
u = F(t)\langle\cos x \sin y, -\sin x \cos y\rangle, \hspace{5mm} p = -\frac{1}{4}F(t)^2\langle\cos 2x + \cos 2y\rangle. 
\end{equation}
Note that setting $F(t) = \exp(- 2\nu t)$ recovers the standard Taylor-Green test from Section \ref{sec:test_nonadaptive}.
Consider the following smooth transition function from zero to one,
\begin{equation}\notag
g(t) = 
\begin{cases}
0 & t\leq 0\\
\exp\left(-\frac{1}{(10t)^{10}}\right) & t>0.
\end{cases}
\end{equation}
This function rapidly approaches one to machine precision, and we construct a periodic $F$ with shifts and translations of $g$, the effect of which is seen in Figure \ref{fig:norms}. 

We tested convergence for the adaptive methods as follows. For five tolerances, \\$\varepsilon = \{10^{-2},10^{-3},10^{-4},10^{-5},10^{-6}\}$, we computed the discrete solutions for BE-AB2, BE-AB2+F, and MOOSE-IMEX-12. We then compare the relative error versus the number of Stokes solves required to complete the simulation, since this is the dominant cost for the methods. Counting total solves is more fair than counting the average $\Delta t$ since adaptive methods reject solutions that do not satisfy the tolerance, which results in additional Stokes solves to recompute the solution with a new $\Delta t_n$. Therefore, the formula for adaptive methods is $$\text{Stokes solves} = \text{accepted solutions + rejected solutions.}$$

Figure \ref{fig:adaptive_error} shows the velocity and pressure errors for both adaptive BE-AB2 and MOOSE-IMEX-12. Even though we include the work of the rejected solves, we still see the predicted convergence rates. Not shown is VSS BE-AB2+F which performed similarly to the full MOOSE-IMEX-12 method.

Next, we show that time adaptivity is needed for this problem. Using the total number of Stokes solves required by the adaptive method for each tolerance, we calculate the effective stepsize as $\Delta t = T/(\text{Stokes solves})$. We then run BE-AB2+F with this fixed stepsize, and compare the error with MOOSE-IMEX-12. The results, shown in Figure \ref{fig:adaptive_vs_nonadaptive}, clearly show that adaptivity is required to solve this problem efficiently. In some cases, MOOSE-IMEX-12 is three orders of magnitude better.

In Figure \ref{fig:norms}, we plot the norms of $u$ and $p$ for the case where both MOOSE-IMEX-12 and BE-AB2+F perform 221 Stokes solves, which corresponds to a tolerance of $\varepsilon = 10^{-2}$. We see that while MOOSE-IMEX-12 essentially captures the transitions, nonadaptive BE-AB2+F exhibits large fluctuations.

Although we currently lack a proof for VSS BE-AB2+F and MOOSE-IMEX-12, our tests indicate both convergence and stability.
\begin{figure}
\centering
\begin{subfigure}{\textwidth}
\includegraphics[width=.495\linewidth]{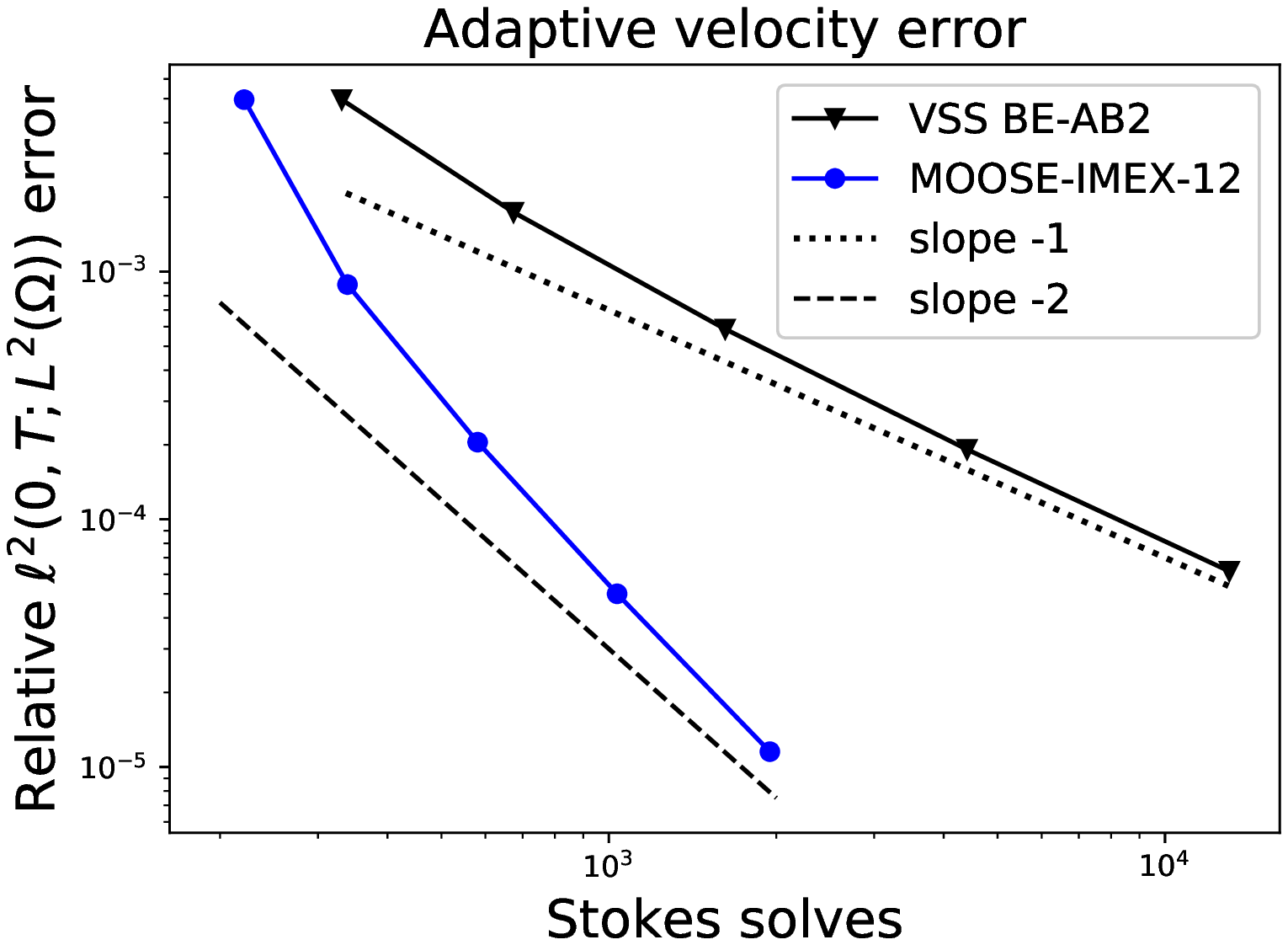}
\includegraphics[width=.495\linewidth]{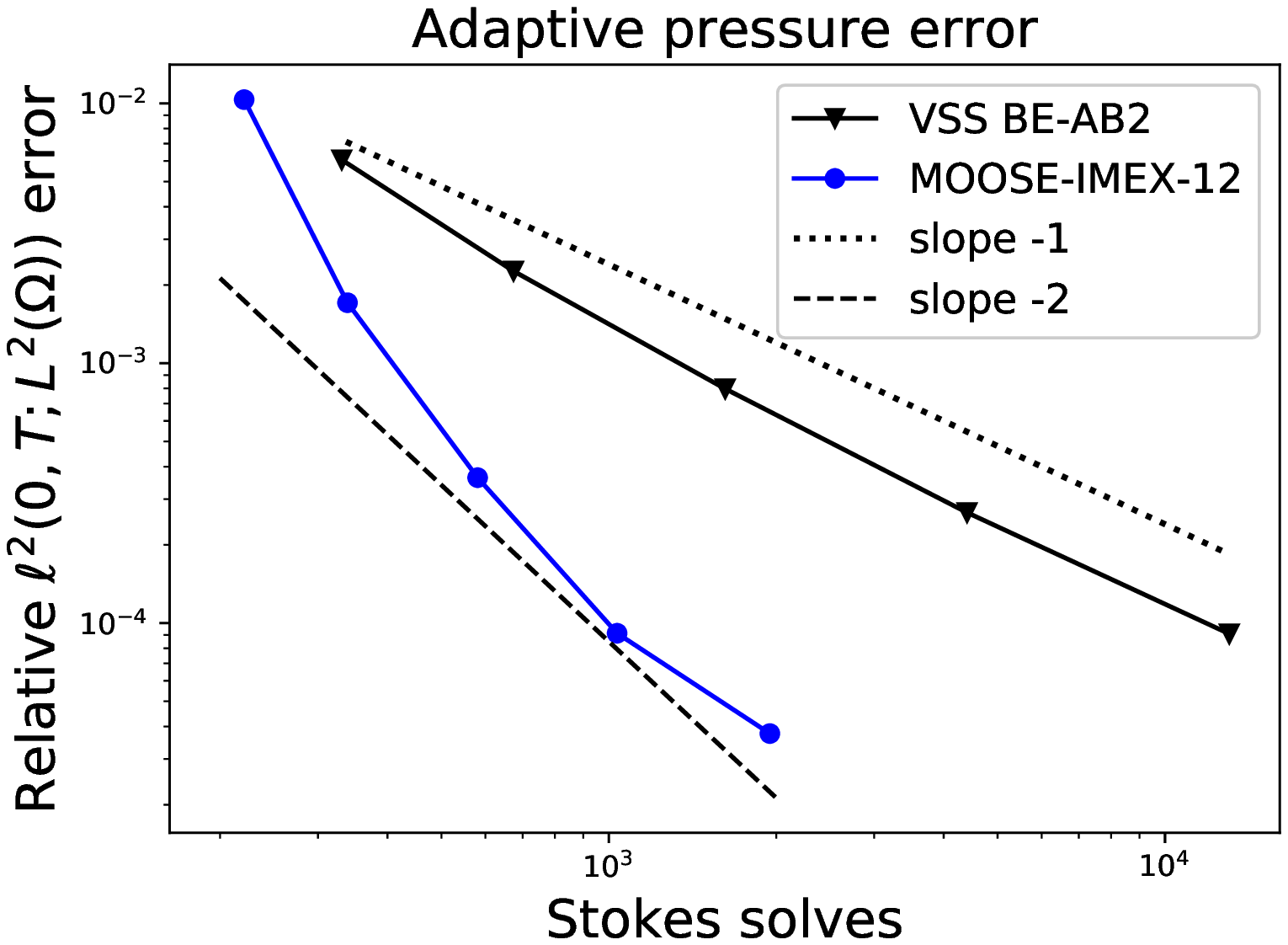}
\end{subfigure}
\caption{The velocity and pressure errors converge at the predicted rates.\label{fig:adaptive_error}}
\end{figure}

\begin{figure}
\centering
\begin{subfigure}{\textwidth}
\includegraphics[width=.495\linewidth]{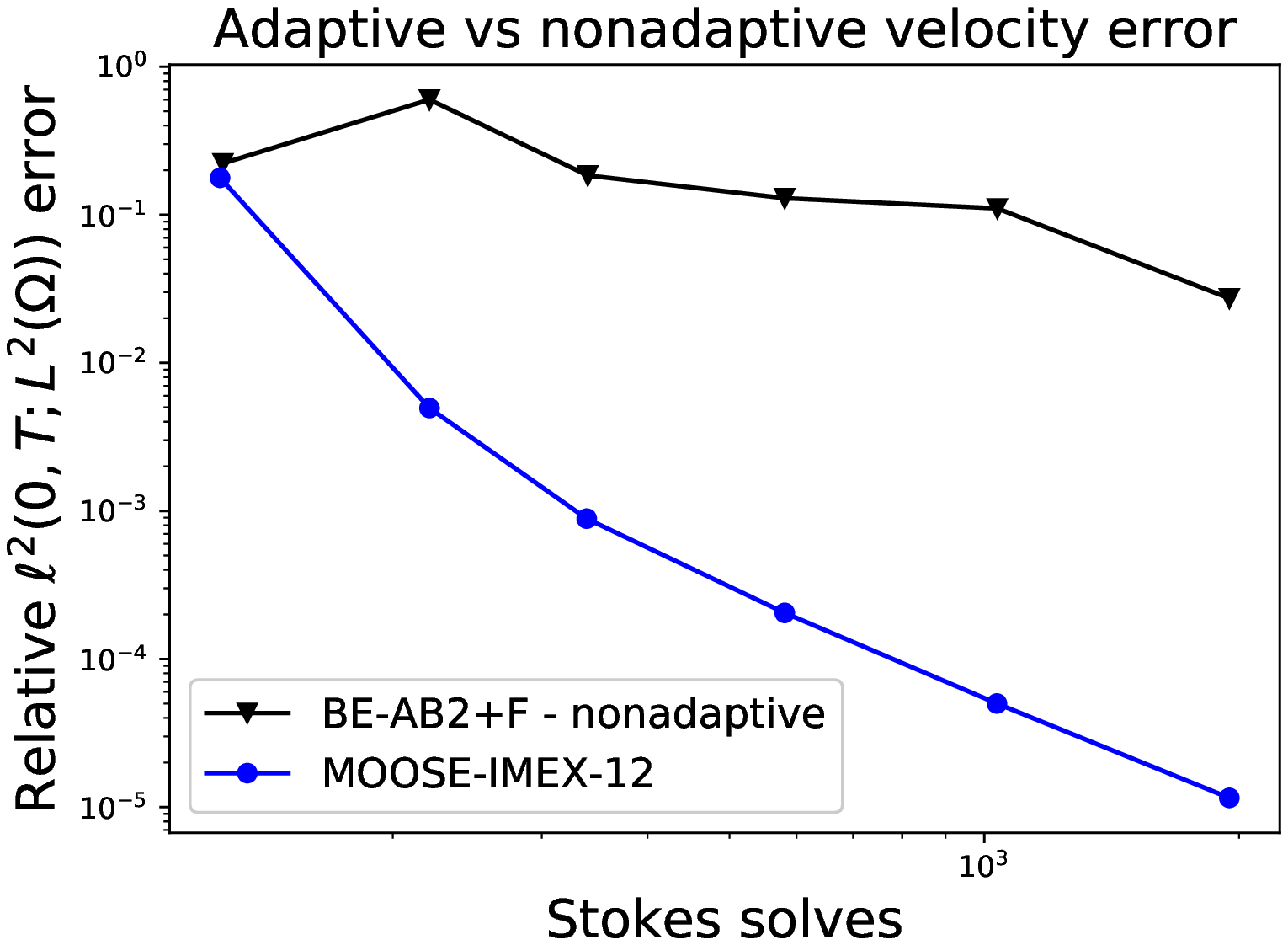}
\includegraphics[width=.495\linewidth]{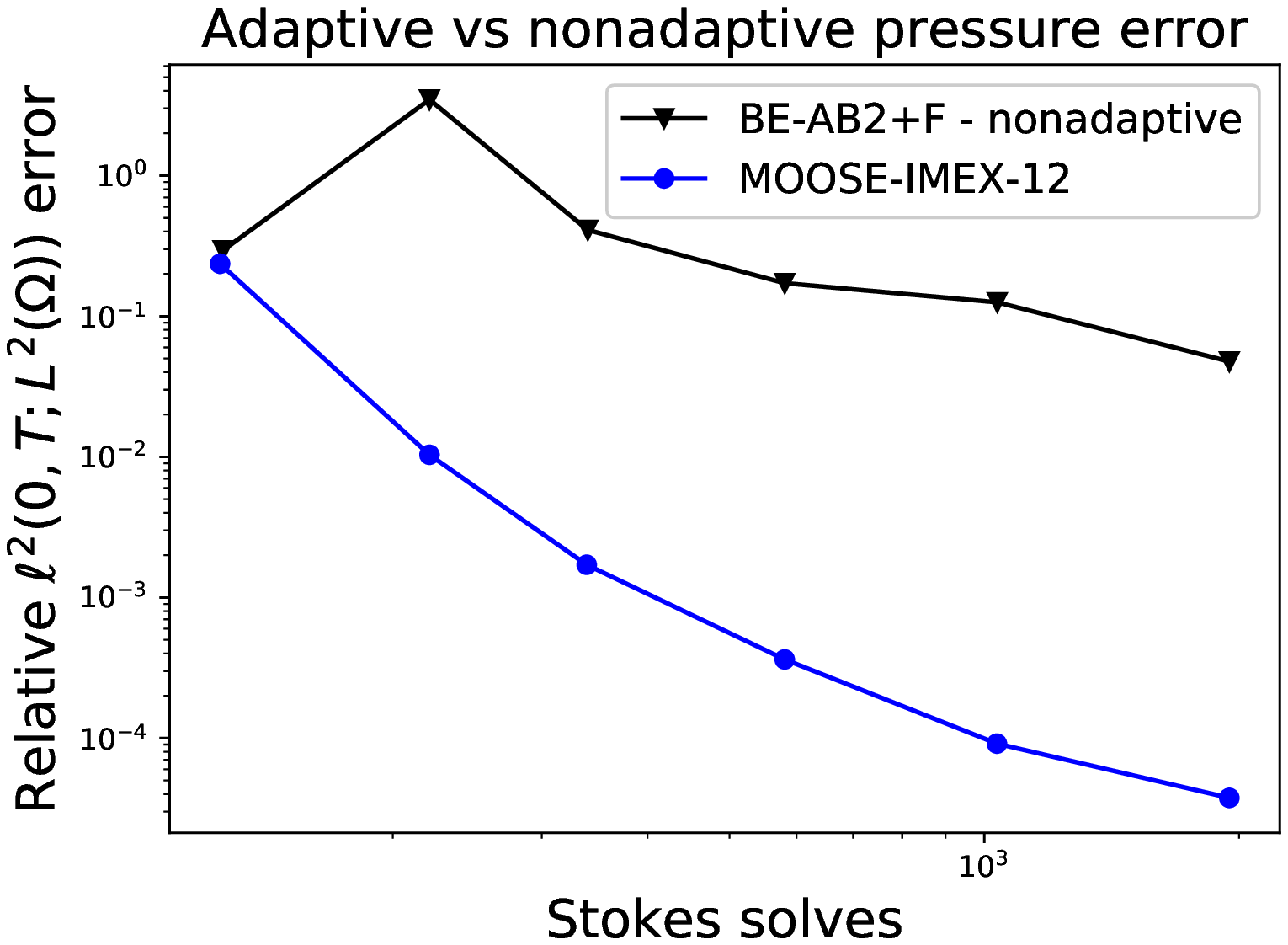}
\end{subfigure}
\caption{While both methods are second order accurate, the adaptive method is orders of magnitude better for the same number of Stokes solves for this test problem.\label{fig:adaptive_vs_nonadaptive}}
\end{figure}

\begin{figure}
\centering
\begin{subfigure}{\textwidth}
\includegraphics[width=.495\linewidth]{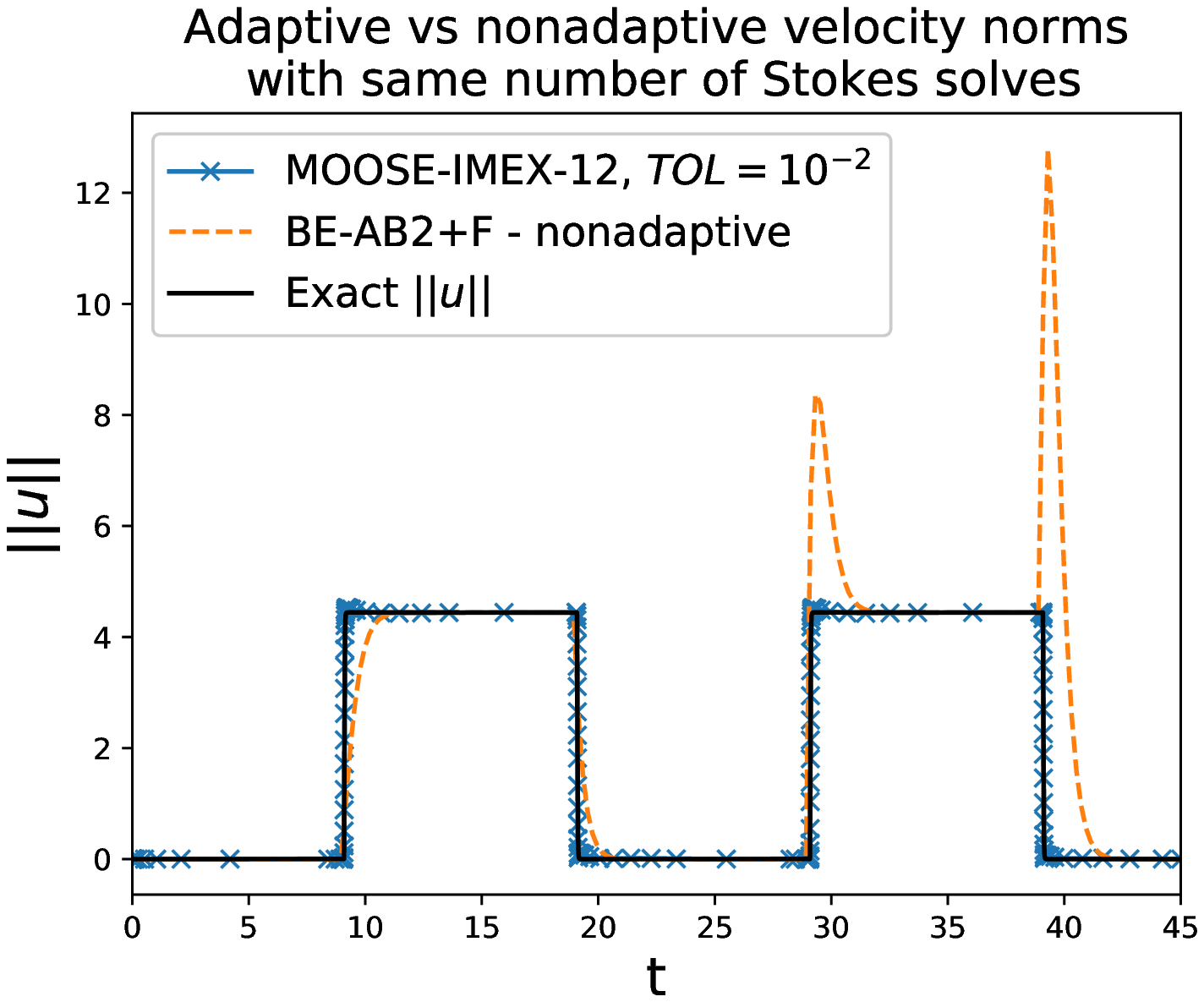}
\includegraphics[width=.495\linewidth]{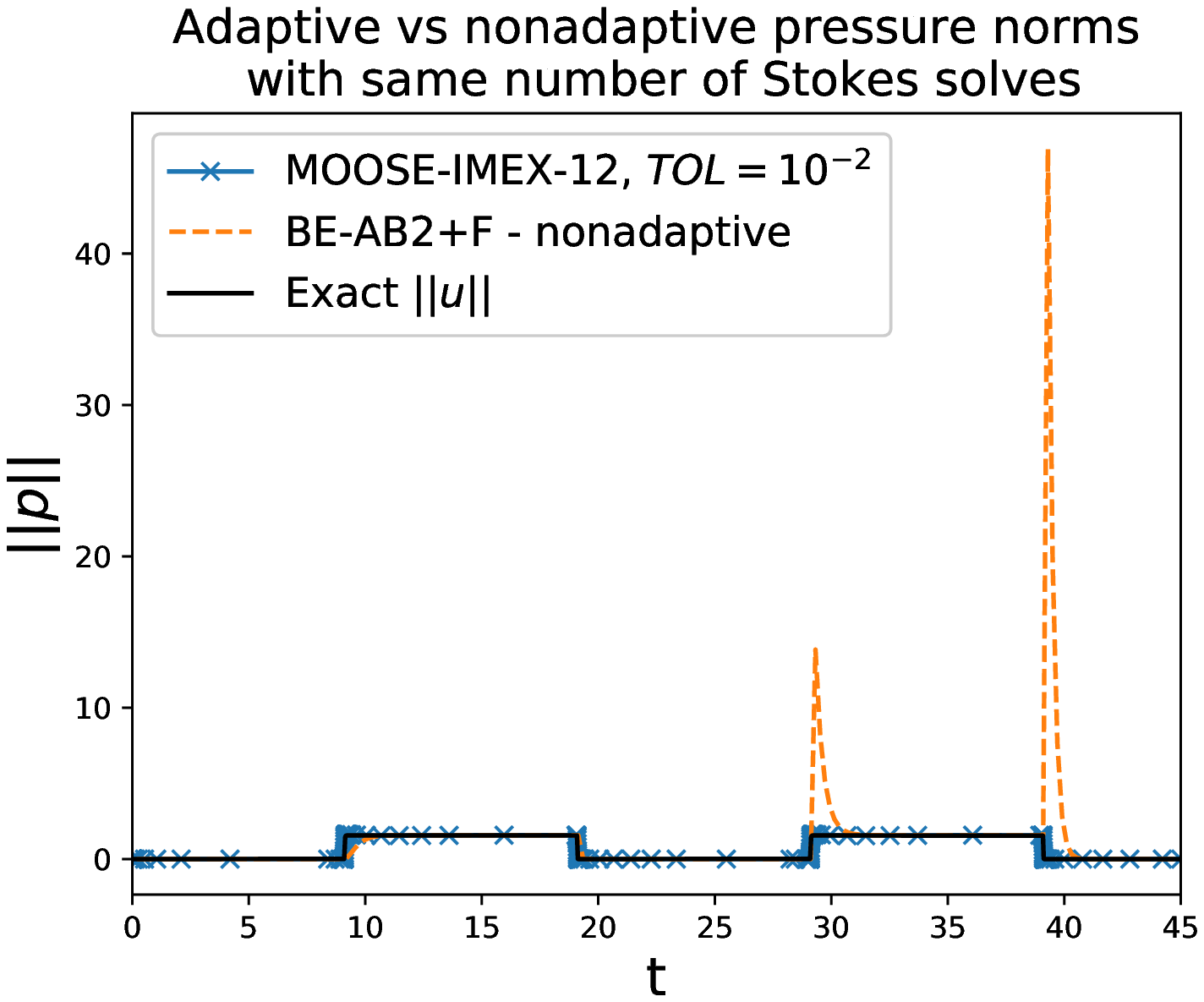}
\end{subfigure}
\caption{For the same number of Stokes solves, the nonadaptive method exhibits overshooting, while the adaptive method resolves transitions.\label{fig:norms}}
\end{figure}

\section{Conclusion\label{sec:conclusion}}
We introduced and analyzed a new variable stepsize IMEX scheme for solving the NSE, BE-AB2. We proved nonlinear energy stability for the variable stepsize method under a timestep and stepsize ratio condition, and without a small data assumption. We are not aware of other proofs of this nature for adaptive, two-step methods for NSE with explicit treatment of the nonlinearity. We then included a full error analysis for the method.

We extended this method to an embedded IMEX pair of orders one (BE-AB2) and two (BE-AB2+F) that requires no additional Stokes solves, and is easy to implement. We prove nonlinear energy stability of constant stepsize BE-AB2+F under a timestep condition. This pair is combined to construct a new variable stepsize, variable order IMEX method for NSE of orders one and two that only requires one Stokes solve per timestep, which we tested herein. We aren't aware of any other such methods.

Future work will consist of higher order extensions of the  MOOSE-IMEX-12 scheme . Based on the methods in \cite{decaria2018new} there appears to a path forward to doing so. Additionally, we will explore the error and stability analysis of the variable stepsize BE-AB2+F method. 

\bibliographystyle{siam}
\bibliography{WorksCited_arxiv}
\end{document}